\documentclass[12pt]{article}

\usepackage[T1]{fontenc}
\usepackage[latin1]{inputenc}
\usepackage{amsmath,amssymb,amscd}
\usepackage{theorem}
\usepackage{graphics}
\usepackage{epic}
\usepackage{epsfig}
\usepackage{subfigure}
\usepackage{bbm}

\usepackage{stmaryrd}

\DeclareFontFamily{U}{txsyc}{}
\DeclareFontShape{U}{txsyc}{m}{n}{
   <-> txsyc%
}{}
\DeclareFontShape{U}{txsyc}{bx}{n}{
   <-> txbsyc%
}{}
\DeclareFontShape{U}{txsyc}{l}{n}{<->ssub * txsyc/m/n}{}
\DeclareFontShape{U}{txsyc}{b}{n}{<->ssub * txsyc/bx/n}{}

\DeclareSymbolFont{symbolsC}{U}{txsyc}{m}{n}
\SetSymbolFont{symbolsC}{bold}{U}{txsyc}{bx}{n}
\DeclareFontSubstitution{U}{txsyc}{m}{n}

\DeclareMathSymbol{\df}{\mathrel}{symbolsC}{"42}
\DeclareMathSymbol{\fd}{\mathrel}{symbolsC}{"43}
\DeclareMathSymbol{\lJoin}{\mathrel}{symbolsC}{"58}
\DeclareMathSymbol{\rJoin}{\mathrel}{symbolsC}{"59}

\newcommand{\cD}{{\cal D}}

\newcommand{\cP}{{\cal P}}

\newcommand{\cT}{{\cal T}}

\newcommand{\CC}{\mathbb{C}}

\newcommand{\LL}{\mathbb{L}}
\newcommand{\NN}{\mathbb{N}}
\newcommand{\PP}{\mathbb{P}}

\newcommand{\RR}{\mathbb{R}}

\newcommand{\ZZ}{\mathbb{Z}}

\newcommand{\iy}{\infty}
\newcommand{\lt}{\left}
\newcommand{\me}{\medskip}

\newcommand{\pa}{\partial}
\newcommand{\ri}{\rightarrow}
\newcommand{\rt}{\right}
\newcommand{\sm}{\smallskip}

\newcommand{\wi}{\widetilde}
\newcommand{\wit}{\widehat}

\newcommand{\lve}{\lt\vert}
\newcommand{\lVe}{\lt\Vert}

\newcommand{\rve}{\rt\vert}
\newcommand{\rVe}{\rt\Vert}

\newcommand{\vvv}{\vert\!\vert\!\vert}

\newcommand{\bq}{\begin{eqnarray*}}
\newcommand{\bqn}[1]{\begin{eqnarray}\label{#1}}
\newcommand{\eq}{\end{eqnarray*}}
\newcommand{\eqn}{\end{eqnarray}}

\newtheorem{thm}{Theorem}[section]
\newtheorem{lem}[thm]{Lemma}
\newtheorem{cor}[thm]{Corollary}

\newtheorem{prop}[thm]{Proposition}
\newtheorem{rem}[thm]{Remark}
\newtheorem{example}{Example}
\newenvironment{proof}{\paragraph{Proof} }{\hfill$\Box$\bigskip}

\theoremstyle{break}

\newtheorem{lembr}[thm]{Lemma}

\newcommand{\lin}{\llbracket}
\newcommand{\rin}{\rrbracket}

\renewcommand{\thepro}{\arabic{pro}}

\title{\bf On Dirichlet eigenvectors for neutral two-dimensional Markov chains}

\author{Nicolas Champagnat\footnote{TOSCA project-team, INRIA Nancy -- Grand Est, IECN, UMR
    7502, Nancy-Universit\'e, Campus Scientifique, B.P.\ 70239, 54506 Vand\oe uvre-l\`es-Nancy
    Cedex, France, E-mail: \texttt{Nicolas.Champagnat@inria.fr}}, Persi
  Diaconis\footnote{Department of Mathematic and Statistics, Sequoia Hall, 390 Serra Mall,
    Stanford University, Stanford, CA 94305-4065s}, Laurent Miclo\footnote{Institut de
    Math\'ematiques de Toulouse, Universit\'e Paul Sabatier, 118 route de Narbonne, F-31062
    Toulouse Cedex 9, France}}

\date{}

\begin{document}

\maketitle

\begin{abstract}
  We consider a general class of discrete, two-dimensional Markov chains modeling the dynamics
  of a population with two types, without mutation or immigration, and neutral in the sense
  that type has no influence on each individual's birth or death parameters. We prove that all
  the eigenvectors of the corresponding transition matrix or infinitesimal generator $\Pi$ can
  be expressed as the product of ``universal'' polynomials of two variables, depending on each
  type's size but not on the specific transitions of the dynamics, and functions depending
  only on the total population size. These eigenvectors appear to be Dirichlet eigenvectors
  for $\Pi$ on the complement of triangular subdomains, and as a consequence the corresponding
  eigenvalues are ordered in a specific way. As an application, we study the quasistationary
  behavior of finite, nearly neutral, two-dimensional Markov chains, absorbed in the sense
  that $0$ is an absorbing state for each component of the process.
\end{abstract}



\section{Introduction}
\label{sec:intro}

This paper studies spectral properties of two-dimensional discrete Markov processes in continuous and discrete
time, having the \emph{neutrality} property, in the sense of population genetics (see
e.g.~\cite{kimura-83}). Considering two populations in interaction, corresponding to two different types of
individuals (typically a mutant type and a resident type), one says that the types are neutral (or the mutant
type is neutral w.r.t.\ the resident type, or more simply the mutation is neutral) if individuals of both
types are indistinguishable in terms of the total population dynamics. In other words, the mutant population
has no selective advantage (or disadvantage) with respect to the rest of the population.

More formally, we say that a two-dimensional Markov process $(X_n,Y_n)_{n\in {\cal T}}$ (where ${\cal
  T}=\ZZ_+ \df \{0,1,\ldots\}$ or $\RR_+$) with values in $\RR_+^2$ or $\ZZ_+^2$ is \textbf{neutral} if
$Z_n=X_n+Y_n$ is a Markov process. In particular, the law of the process $Z$ depends on $Z_0$, but not on
$X_0$ or $Y_0$.

If the process $Z$ is a birth and death continuous-time chain, the class of neutral processes we consider is
the following: the birth and death rates of the Markov process $(Z_t)_{t\in\RR_+}$ when $Z$ is in state $k\geq
0$ are of the form $k\lambda_k$ and $k\mu_k$, respectively. Note that $0$ is an absorbing state for $Z$. With
this notation, the parameters $\lambda_k$ and $\mu_k$ can be interpreted as birth and death rates \emph{per
  individual}. Then the process $(X_t,Y_t)_{t\in\RR_+}$ is the birth and death process where both types of
individuals have birth and death rates per individual $\lambda_k$ and $\mu_k$, when the total population has
$k$ individuals. This leads to the following transition rates for the Markov process $(X_t,Y_t)_{t\in\RR_+}$:
for all $(i,j)\in\ZZ_+^2$,
\begin{equation*}
  \begin{array}{lll}
    \mbox{from\ }(i,j)\mbox{\ to\ }(i+1,j) & \mbox{with rate} & i\:\lambda_{i+j} \\
    \mbox{from\ }(i,j)\mbox{\ to\ }(i,j+1) & \mbox{with rate} & j\:\lambda_{i+j} \\
    \mbox{from\ }(i,j)\mbox{\ to\ }(i-1,j) & \mbox{with rate} & i\:\mu_{i+j} \\
    \mbox{from\ }(i,j)\mbox{\ to\ }(i,j-1) & \mbox{with rate} & j\:\mu_{i+j}. \\
  \end{array}
\end{equation*}
Note that the sets $\{0\}\times\ZZ_+$, $\ZZ_+\times\{0\}$ and $\{(0,0)\}$ are absorbing for this process. In
other words, we only consider neutral two-dimensional processes \textbf{without mutation and
  immigration}.

In the case of discrete time, we consider two-dimensional birth and death processes constructed in a similar
way: assume that the birth and death probabilities of the process $(Z_n)_{n\in\ZZ_+}$ when in state $k$ are
$p_k$ and $q_k$, respectively, with $p_k+q_k\leq 1$. Then, when a birth or a death occurs in the population,
the individual to which this event applies is chosen uniformly at random in the population. This leads to the
transition probabilities
\begin{equation*}
  \begin{array}{lll}
    \mbox{from\ }(i,j)\mbox{\ to\ }(i+1,j) & \mbox{with probability} & \frac{i}{i+j}\:p_{i+j} \\
    \mbox{from\ }(i,j)\mbox{\ to\ }(i,j+1) & \mbox{with probability} & \frac{j}{i+j}\:p_{i+j} \\
    \mbox{from\ }(i,j)\mbox{\ to\ }(i-1,j) & \mbox{with probability} & \frac{i}{i+j}\:q_{i+j} \\
    \mbox{from\ }(i,j)\mbox{\ to\ }(i,j-1) & \mbox{with probability} & \frac{j}{i+j}\:q_{i+j} \\
    \mbox{from\ }(i,j)\mbox{\ to\ }(i,j) & \mbox{with probability} & r_k,
  \end{array}
\end{equation*}
where $r_k \df 1-p_k-q_k$. Note that this construction requires assuming that $r_0=1$ (i.e.\ that $0$ is
absorbing for $Z$).

In~\cite{karlin-mcgregor-75}, Karlin and McGregor studied two families of neutral multitype
population processes (branching processes and Moran model), but only in the case of nonzero
mutation or immigration, for which the set of states where one population (or more) is extinct
is not absorbing. They could express the eigenvectors of the corresponding infinitesimal
generators in terms of Hahn polynomials. We focus here on neutral processes without mutation
and immigration, which are singular for the approach of~\cite{karlin-mcgregor-75}, and we
apply our study to a much bigger class of neutral population processes, containing the birth
and death processes described above, but also non-birth and death models.

Our main result is the characterization of all eigenvalues and right eigenvectors of the transition matrix of
neutral processes without mutation and immigration. To this aim, we first consider the (easier) continuous state space case in
Section~\ref{sec:cont} to introduce some tools used in the sequel. Next, we construct a particular family of
polynomials of two variables in Section~\ref{sec:poly}, using linear algebra arguments. In
Section~\ref{sec:diago}, we prove that the eigenvectors of the transition matrix of neutral two-dimensional
Markov processes can be decomposed as the product of ``universal'' polynomials (in the sense that they do not
depend on the specific transition rates of the Markov chain) with functions depending only on the total
population size. 
We then relate these eigenvectors with Dirichlet eigenvalue problems in subdomains
of $\ZZ_+^2$ of the form $\{(i,j)\in\NN^2:i+j\geq k\}$ for $k\geq 2$ (Section~\ref{sec:Dirichlet}), where
$\NN=\{1,2,\ldots\}$.

The last section (Section~\ref{sec:QSD}) is devoted to the application of the previous results to the study of
quasi-stationary distributions. A probability distribution $\nu$ on $\ZZ_+^2\setminus\{0\}$ is called
\emph{quasi-stationary} if it is invariant conditionally on the non-extinction of the whole population, i.e.\
if
$$
\PP_\nu((X_1,Y_1)=(i,j)\mid Z_1\not=0)=\nu_{i,j},\quad\forall (i,j)\in\ZZ_+^2\setminus\{0\},
$$
where $\PP_\nu$ denotes the law of the process $(X,Y)$ with initial distribution $\nu$. This
question is
related to the notion of \emph{quasi-limiting distribution} (also called ``Yaglom limit'', in reference to
Yaglom's theorem on the same convergence for Galton-Watson processes), defined as
$$
\nu_{i,j} \df \lim_{n\rightarrow+\infty}\PP((X_n,Y_n)=(i,j)\mid Z_n\not=0),\quad\forall
(i,j)\in\ZZ_+^2\setminus\{0\}.
$$
These notions are relevant in cases where extinction occurs almost surely in finite time, to describe the
``stationary behaviour'' of the process before extinction when the extinction time is large. This is typically
the case in many population dynamics models, where ecological interactions in the population produce high
mortality only when the population size is large (one speaks of density-dependent models, see
e.g.~\cite{murray-93}, or~\cite{champagnat-lambert-07} for discrete stochastic models).

These questions have been extensively studied in the case where the transition matrix restricted to the
non-extinct states is irreducible (which is not true in our two-dimensional case). The first paper of
Darroch and Seneta~\cite{darroch-seneta-65} studies the discrete-time, finite case. Several extensions of
these results to continuous-time and/or infinite denumerable state spaces have then been considered
in~\cite{seneta-verejones-66,darroch-seneta-67,flaspohler-74}. The case of population dynamics in dimension 1
have been studied by many authors
(e.g.~\cite{cavender-78,vandoorn-91,kijima-seneta-91,ferrari-martinez-al-92,kesten-95,hognas-97,nasell-99,gosselin-01}).
More
recently, the quasi-stationary behaviour of one-dimensional diffusion models has been studied
in~\cite{cattiaux-collet-al-09}. As far as we know, the two-dimensional case has only been studied in the
continuous state space (diffusion) case~\cite{cattiaux-meleard-10}. An extensive bibliography
on quasi-stationary distributions can be found in~\cite{pollett-11}

In Subsection~\ref{sec:Yaglom-general}, we first give the quasi-limiting distribution for general finite
two-dimensional Markov chains in terms of the maximal Dirichlet eigenvalues of the transition matrix in
several subdomains. Finally, in Subsection~\ref{sec:Yaglom-neutral}, we apply our previous results  to prove
that coexistence in the quasi-limiting distribution is impossible for two-dimensional finite Markov chains
which are close to neutrality.  

The paper ends with a glossary of all the notation used in Sections~\ref{sec:diago}
to~\ref{sec:QSD}, which may appear at different places in the paper.

\section{Preliminary: continuous case}
\label{sec:cont}

In this section, we consider the continuous state space case, where computations are easier, in order to
introduce some of the tools needed in the discrete case.

Fix $p$ and $q$ two measurable functions from $\mathbb{R}_+$ to $\mathbb{R}_+$, and consider the system of
stochastic differential equations
\begin{subequations}
  \label{eq:SDE}
  \begin{equation}
    \label{eq:SDE-1}
      dX_t=\sqrt{2X_tp(X_t+Y_t)}dB^1_t+X_tq(X_t+Y_t)dt
  \end{equation}
  \begin{equation}
    \label{eq:SDE-2}
    dY_t=\sqrt{2Y_tp(X_t+Y_t)}dB^2_t+Y_tq(X_t+Y_t)dt,    
  \end{equation}
\end{subequations}
where $(B^1,B^2)$ is a standard two-dimensional Brownian motion. Such SDEs are sometimes called branching
diffusions, and are biologically relevant extensions of the classical Feller
diffusion~\cite{cattiaux-collet-al-09}. If $p$ and $q$ satisfy appropriate growth and regularity assumptions,
the solution to this system of SDEs is defined for all positive times and can be obtained as scaling limits of
two-dimensional birth and death processes (we refer to~\cite{cattiaux-collet-al-09} for the one-dimensional
case; the extension to higher dimensions is easy).

This process is neutral in the sense defined in the introduction since $Z_t=X_t+Y_t$ solves the SDE
\begin{equation}
  \label{eq:EDS-Z}
  dZ_t=Z_tq(Z_t)dt+\sqrt{2Z_tp(Z_t)}dB_t,
\end{equation}
where 
$$
B_t=\int_0^t\sqrt{\frac{X_s}{Z_s}}dB^1_s+\int_0^t\sqrt{\frac{Y_s}{Z_s}}dB^2_s
$$
is a standard Brownian motion. Note also that $\RR_+\times\{0\}$, $\{0\}\times\RR_+$ and $\{(0,0)\}$ are
absorbing states as soon as there is uniqueness in law for the system~(\ref{eq:SDE}).

For any $\varphi\in C^2(\RR_+^2)$, the infinitesimal generator $A$ of the process $(X_t,Y_t)_{t\geq 0}$ is
given by
\begin{align}
  A\varphi(x,y) & =\left(x\frac{\partial^2\varphi}{\partial x^2}(x,y)
    +y\frac{\partial^2\varphi}{\partial y^2}(x,y)\right)p(x+y) \notag \\
  & +\left(x\frac{\partial\varphi}{\partial x}(x,y)
    +y\frac{\partial\varphi}{\partial y}(x,y)\right)q(x+y).
  \label{eq:continu}
\end{align}
We first observe in the following proposition that $A$ admits a symmetric measure, but only on a subset of
$C^2(\RR_+^2)$.

\begin{prop}
  \label{prop:mu-SDE}
  Assume that $1/p$ and $q/p$ belong to $\LL^1_{\text{loc}}((0,+\infty))$. Let us define the measure
  $\mu$ on $\RR_+^2$ as
  \begin{equation}
    \label{eq:def-mu-SDE}
    \mu(dx,dy)=\frac{\exp\left(\int_{1}^{x+y}\:\frac{q(s)}{p(s)}ds\right)}{x\:y\:p(x+y)}\:dx\:dy.
  \end{equation}
  Then, the restriction $\wi{A}$ of the operator $A$ to $C^2_c((0,+\infty)^2)$ is symmetric for the
  canonical inner product $\langle\cdot,\cdot\rangle_\mu$ in $\LL^2(\RR_+^2,\mu)$, and, hence, so is its
  closure in $\LL^2(\RR^2_+,\mu)$.
\end{prop}

Note that, because of the singularities in $\mu$ when $x$ or $y$ vanish, if $p\geq c>0$ in the
neighborhood of 0, any continuous function in $\LL^2(\RR_+^2,\mu)$ must vanish at the boundary of
$\RR_+^2$. Therefore, $\LL^2(\RR_+^2,\mu)\subset\LL^{2,0}_{\text{loc}}(\RR_+^2)$, where
$\LL^{2,0}_{\text{loc}}(\RR_+^2)$ is
defined as the closure of $C_c((0,+\infty)^2)$ in
$\LL^2_{\text{loc}}(\RR_+^2)$.

\begin{proof}
For all $f,g\in C^2_c((0,+\infty)^2)$, we have (formally)
\begin{align*}
  \langle f,Ag\rangle_\mu & =-\int_{\RR_+^2}p(x+y)\left(x\frac{\partial f}{\partial x}(x,y)\frac{\partial
      g}{\partial y}(x,y)+y\frac{\partial f}{\partial y}(x,y)\frac{\partial
      g}{\partial y}(x,y)\right)\mu(x,y)\:dx\:dy \\
  & -\int_{\RR_+^2}f(x,y)\left(\frac{\partial(x\mu p)}{\partial x}(x,y)\frac{\partial g}{\partial x}(x,y)
    +\frac{\partial(y\mu p)}{\partial y}(x,y)\frac{\partial g}{\partial y}(x,y)\right)\:dx\:dy \\
  & +\int_{\RR_+^2}f(x,y)q(x+y)\left(x\frac{\partial g}{\partial x}(x,y)
    +y\frac{\partial g}{\partial y}(x,y)\right)\mu(x,y)\:dx\:dy.
\end{align*}
Therefore, $\langle f,Ag\rangle_\mu=\langle Af,g\rangle_\mu$ if
$$
xq(x+y)\mu(x,y)=\frac{\partial (x\mu p)}{\partial x}(x,y),\quad\forall x,y>0
$$
and
$$
yq(x+y)\mu(x,y)=\frac{\partial (y\mu p)}{\partial y}(x,y),\quad\forall x,y>0.
$$
Conversely, these equalities can be directly checked from the formula~(\ref{eq:def-mu-SDE}), which implies
that $\partial(x\mu p)/\partial x$ and $\partial(y\mu p)/\partial y$ exist in a weak sense.
\end{proof}

Before studying the eigenvectors of $A$, we need the following result.
\begin{prop}
  \label{prop:poly-ortho}
  For all $\lambda\in\RR$, the problem
  \begin{equation}
    \label{eq:ODE-h}
    (1-x^2)h''(x)=-\lambda h(x)
  \end{equation}
  has no (weak) non-zero solution $h\in C^1([-1,1])$ except when $\lambda=d(d-1)$ for some $d\in\NN$. For
  $d=1$, the vector space of solutions has dimension 2 and is spanned by the two polynomials $h(x)=1$ and
  $h(x)=x$.  For all $d\geq 2$,
  \begin{equation}
    \label{eq:ODE-Hd}
    (1-x^2)h''(x)+d(d-1)h(x)=0
  \end{equation}
  has a one-dimensional vector space of solutions in $C^1([-1,1])$, spanned by a polynomial $H_d$ of degree
  $d$, which can be chosen such that the family $(H_d)_{d\geq 2}$ is an orthonormal basis of
  $\LL^2([-1,1],\frac{dx}{1-x^2})$. In addition,
  $H_d$ has the same parity as $d$ (all powers of $x$ that appear are even if $d$ is even and
  odd if $d$ is odd).
\end{prop}

\begin{proof}
The case $\lambda=0$ is trivial, so let us assume that $\lambda\not=0$. Note that, assuming $\lambda\not=0$
implies that any solution $h\in C^1([-1,1])$ satisfying~(\ref{eq:ODE-h}) must satisfy
$h(-1)=h(1)=0$. Indeed, the function $h(x)/(1-x^2)$ has a finite integral over $[-1,1]$ in
that case only. The
equation has the form of the classical Sturm-Liouville problem, but does not satisfy the usual integrability
conditions, since $(1-x^2)^{-1}$ is not integrable on $[-1,1]$.

First, uniqueness in $C^1([-1,1])$ for~(\ref{eq:ODE-h}) can be proved using the classical Wronskian
determinant: given two solutions $h_1$ and $h_2$ of~(\ref{eq:ODE-h}) in $C^1([-1,1])$, set
$W(x) \df h_1(x)h'_2(x)-h'_1(x)h_2(x)$. Since $h_i(-1)=h_i(1)=0$ for $i=1,2$, we have $W(-1)=W(1)=0$. Moreover,
one has $W'(x)=0$ a.e.\ in $[-1,1]$ and thus $W(x)=0$ on $[-1,1]$. Thus $(h_1(0),h'_1(0))$ and
$(h_2(0),h'_2(0))$ are linearly dependent, and so $h_1$ and $h_2$ as well.

Second, assuming that~(\ref{eq:ODE-h}) has a polynomial solution of degree $d$, the comparison of the higher
degree terms yields $\lambda=d(d-1)$. So assume that $\lambda=d(d-1)$ with $d\geq 2$. The polynomial
\begin{equation}
  \label{eq:etoile}
  h(x)=\sum_{0\leq k\leq d,\ d-k\text{\ even}}a_{d,k}x^k,  
\end{equation}
where $a_{d,d}\not=0$ and
$$
(k(k-1)-d(d-1))a_{d,k}=(k+2)(k+1)a_{d,k+2},\quad\forall 0\leq k\leq d-2,\quad d-k\text{\ even},
$$
is a nonzero solution of~(\ref{eq:ODE-Hd}). Since $h(1)=h(-1)=0$, we can choose $a_{d,d}$ to define the
polynomial $H_d \df h$ satisfying $\|H_d\|^2_{\rho}=1$, where $\rho(dx)=(1-x^2)^{-1}dx$ and $\|\cdot\|_\rho$ is
the norm corresponding to the inner product $\langle\cdot,\cdot\rangle_\rho$.

Third, let $h\in C^2([-1,1])$ be a solution to~(\ref{eq:ODE-h}) for some $\lambda\not=0$. Then $h(x)/(1-x^2)$
is continuous on $[-1,1]$, and for any $d\geq 2$,
\begin{equation}
  \label{eq:pf-ortho}
  \begin{aligned}
    \int_{-1}^1\frac{h(x)}{1-x^2}H_d(x)dx & =\frac{-1}{d(d-1)}\int_{-1}^1h(x)H_d''(x)dx \\ &
    =\frac{\lambda}{d(d-1)}\int_{-1}^1\frac{h(x)}{1-x^2}H_d(x)dx,
  \end{aligned}
\end{equation}
where the integration by parts is justified by the Dirichlet conditions for $h$ and $H_d$. Therefore, either
$\lambda=d(d-1)$ for some $d\geq 2$, and then $f$ and $H_d$ are linearly dependent by the previous uniqueness
result, or
$$
\langle h,H_d\rangle_\rho=\int_{-1}^{1}h(x)\frac{H_d(x)}{1-x^2}\:dx=0,\quad \forall d\geq 2,
$$
where $\{H_d(x)/(1-x^2)\}_{d\geq 2}$ is a family of polynomials of all nonnegative degrees. Therefore,
$\int_{-1}^1 h(x)P(x)dx=0$ for all polynomial $P$, and hence $h=0$.

The fact that the polynomials $\{H_d\}_{d\geq 2}$ are orthogonal w.r.t.\ the inner product
$\langle\cdot,\cdot\rangle_\rho$ follows from~(\ref{eq:pf-ortho}) with $h=H_{d'}$ for $d'\not=d$.

Finally, assume that $h\in\LL^2([-1,1],\rho)$ satisfies $\langle h,H_d\rangle_\rho=0$ for all $d\geq 2$.  As
above, we then have
$$
\int_{-1}^1 h(x)P(x)\:dx=0
$$
for all polynomial $P$. Since $h\in\LL^2([-1,1],\rho)\subset\LL^2([-1,1],dx)$, such relations imply that $h=0$
a.e., which ends the proof that $\{H_d\}_{d\geq 2}$ is an orthonormal basis of $\LL^2([-1,1],\rho)$.
The parity property of $H_d$ follows from~(\ref{eq:etoile}).
\end{proof}

\begin{rem}
  \label{rem:Gegenbauer}
  The polynomials $H_d$ may be obtained as a limit case of Gegenbauer polynomials, defined for
  a parameter $\lambda>-1/2$ as (up to a multiplicative constant)
  $$
  H^\lambda_d(x)=\sqrt{B\Big(\frac{1}{2},\lambda+\frac{1}{2}\Big)\frac{2(2\lambda+1)_d(d+\lambda)}{d!(d+2\lambda)}}{_2}F_1\biggl(\begin{matrix}-d,\
    d+2\lambda\\\lambda+\frac{1}{2}\end{matrix};\frac{1+x}{2}\biggr),
  $$
  for all $d\geq 0$, where $B$ is the Beta function, ${_a}F_b$ are the classical
  hypergeometric series and $(x)_n:=x(x+1)\ldots(x+n-1)$ is the shifted factorial. The
  normalizing constant has been chosen such that these polynomials form an orthonormal family
  in $\mathbb{L}^2((-1,1),(1-x^2)^{\lambda-1/2}dx)$. Then, for all $d\geq 2$, $H_d$ may be
  obtained as
  \begin{equation}
    \label{eq:Gegen}
    H_d(x)=\lim_{\lambda\rightarrow -1/2}H^\lambda_d(x)
    =\sqrt{\frac{2d-1}{d(d-1)}}\sum_{k=1}^d\frac{(-d)_k(d-1)_k}{k!(k-1)!}\;\frac{(1+x)^k}{2^{k-1}}.
  \end{equation}
  One easily checks from this formula that $H_d$ solves~(\ref{eq:ODE-Hd}), and the fact that
  $\int_{-1}^1H_d(x)^2(1-x^2)^{-1}dx=1$ is obtained applying the limit
  $\lambda\rightarrow-1/2$ in $\int_{-1}^1 H^\lambda_d(x)^2 (1-x^2)^{\lambda-1/2}dx=1$.
\end{rem}

We now introduce the change of variables
\begin{equation}
  \label{eq:var-z-w}
  z=x+y\in\RR_+\quad\mbox{and}\quad w=\frac{x-y}{x+y}\in[-1,1],
\end{equation}
which defines a $C^\infty$-diffeomorphism from $\RR_+^2\setminus\{0\}$ to $(0,+\infty)\times[-1,1]$. Then
$\mu(dx,dy)=\nu(dz,dw)$, where
$$
\nu(dz,dw) \df \frac{2\exp\left(\int_{1}^{z}\:\frac{q(s)}{p(s)}ds\right)}{z(1-w^2)p(z)}\:dz\:dw,
$$
and
\begin{equation}
  \label{eq:L-CV}
  A=\frac{1-w^2}{z}p(z)\frac{\partial^2}{\partial w^2}+zp(z)\frac{\partial^2}{\partial
    z^2}+zq(z)\frac{\partial}{\partial z}.  
\end{equation}

Now, assume that $A$ has an eigenvector of the form $\varphi(w)\psi(z)$. If $\lambda$ is the corresponding
eigenvalue,~(\ref{eq:L-CV}) yields
$$
\varphi(w)\Big(zp(z)\psi''(z)+zq(z)\psi'(z)-\lambda\psi(z)\Big)+(1-w^2)\varphi''(w)\frac{p(z)}{z}\psi(z)=0.
$$
Hence, we have the following result.
\begin{prop}
  \label{prop:eigen-vect-cont}
  All functions on $\RR_+^2$ of the form
  \begin{equation}
    \label{eq:eigenv-A}
    H_d\Big(\frac{x-y}{x+y}\Big)\psi(x+y),    
  \end{equation}
  where $d\geq 0$, $H_1(x)=x$, $H_0(x)=1$, $H_k$, $k\geq 2$ are as in
  Proposition~\ref{prop:poly-ortho}, where $\psi$ satisfies
  \begin{equation}
    \label{eq:EDO-vp-2}
    zp(z)\psi''(z)+zq(z)\psi'(z)-d(d-1)\frac{p(z)}{z}\psi(z)=\lambda\psi(z),\quad\forall z\geq 0,
  \end{equation}
  for some $\lambda\in\RR$, are eigenfunctions of $A$ for the eigenvalue $\lambda$.
\end{prop}

Now, we proceed to prove that there are no other eigenvectors of $A$ (in the sense of
Theorem~\ref{thm:eigen-SDE} below). Such a result seems natural as~(\ref{eq:EDO-vp-2}) can be written as the
Sturm-Liouville problem
$$
\frac{d}{dz}\left(\psi'(z)\exp\int_1^z\frac{q(s)}{p(s)}ds\right)
-d(d-1)\frac{\exp\int_1^z\frac{q(s)}{p(s)}ds}{z^2}\psi(z)
=-\lambda\frac{\exp\int_1^z\frac{q(s)}{p(s)}ds}{zp(z)}\psi(z).
$$
Here again, the usual integrability conditions are not satisfied. More precisely, if $p(z)\geq c>0$ for $z$ in
the neighborhood of 0 and $q$ is bounded, using the terminology of~\cite{zettl-05}, for all $d\geq 0$, the
problem~(\ref{eq:eigenv-A}) is a singular self-adjoint boundary value problem on $(0,\infty)$, where the
endpoint 0 is LP (limit point) singular (see~\cite[Thm.\:7.4.1]{zettl-05}). In this case very little is known
about the existence of an orthonormal basis of eigenvectors in $\LL^2((0,\infty),\wi{\nu})$, where
\begin{equation}
  \label{eq:def-tilde-nu}
  \wi{\nu}(z)=\frac{2\exp\left(\int_1^z\frac{q(s)}{p(s)}ds\right)}{zp(z)}  
\end{equation}
(the spectrum might even be continuous, see~\cite[Thm.\:10.12.1(8)]{zettl-05}).

For this reason, we state our next result on the operator $A$ with a restricted domain corresponding to the
case where the diffusion is reflected in the set
$$
D \df \{(x,y)\in\RR_+^2:a\leq x+y\leq b\},
$$
where $0<a<b<\infty$. For all initial condition $(X_0,Y_0)=(x_0,y_0)\in D$, we consider the process
$(X_t,Y_t,k^{(a)}_t,k^{(b)}_t)_{t\geq 0}$ such that $k^{(a)}_0=k^{(b)}_0=0$, $k^{(a)}$ and $k^{(b)}$ are
nondecreasing processes, $(X_t,Y_t)\in D$ for all $t\geq 0$,
\begin{gather*}
  k^{(\alpha)}_t=\int_0^t\mathbbm{1}_{\{X_s+Y_s=\alpha\}}dk^{(\alpha)}_s,\quad\forall t \geq 0,\quad \alpha=a,b \\
  \begin{aligned}
    dX_t & =\sqrt{2X_tp(X_t+Y_t)}dB^1_t+X_tq(X_t+Y_t)dt-\sqrt{2}^{-1}dk^{(b)}_t+\sqrt{2}^{-1}dk^{(a)}_t \\
    dY_t & =\sqrt{2Y_tp(X_t+Y_t)}dB^2_t+Y_tq(X_t+Y_t)dt-\sqrt{2}^{-1}dk^{(b)}_t+\sqrt{2}^{-1}dk^{(a)}_t.
  \end{aligned}
\end{gather*}
Then $Z_t=X_t+Y_t$ is the solution of~(\ref{eq:EDS-Z}) reflected at $a$ and $b$ with local time $k_t^{(a)}$ at
$a$ (resp.\ $k^{(b)}_t$ at $b$). Therefore, $(X_t,Y_t)$ is also neutral in the sense of the introduction.

The corresponding infinitesimal generator is defined by~(\ref{eq:continu}) with domain the set of
$\varphi(x,y)\in C^1(D)\cap C^2(\text{int}(D))$, where $\text{int}(D)$ denotes the interior of $D$, such that
$$
\frac{\partial\varphi}{\partial x}(x,y)+\frac{\partial\varphi}{\partial y}(x,y)=0,\quad\forall (x,y)\in D
\text{\ s.t.\ }x+y=a\text{\ or\ }b.
$$

\begin{thm}
  \label{thm:eigen-SDE}
  For $0<a<b<\infty$, assume that $p\geq c>0$ on $[a,b]$ and $q/p$ belong to $\LL^1([a,b])$.
  \begin{description}
  \item[\textmd{(a)}] There exists a denumerable orthonormal basis of $\LL^2(D,\mu)$ of eigenvectors of $A$ of
    the form~(\ref{eq:eigenv-A}), where $d\geq 2$ and $\psi$ solves~(\ref{eq:EDO-vp-2}) on $(a,b)$ and
    satisfies $\psi'(a)=\psi'(b)=0$. Moreover, any eigenvector of $A$ in $\LL^2(D,\mu)$ is a linear
    combination of eigenvectors of the form~(\ref{eq:eigenv-A}), all corresponding to the same eigenvalue.
  \item[\textmd{(b)}] There exists a family of right eigenvectors of $A$ of the form
    \begin{multline}
      \{1\}\cup\Big\{\frac{x}{x+y}\psi_k(x+y),\:
      \frac{y}{x+y}\psi_k(x+y)\Big\}_{k\geq 1}\cup \\ 
      \bigcup_{d\geq 2}\Big\{H_d\Big(\frac{x-y}{x+y}\Big)\psi^{(d)}_k(x+y)\Big\}_{k\geq 1}, \label{eq:basis}
    \end{multline}
    which is a basis of the vector space
    \begin{multline}
      V \df \Big\{f\in\LL^2(D,\textup{Leb}):\exists f_1,f_2\in\LL^2([a,b],\textup{Leb})\text{\ and\
      }f_3\in\LL^2(D,\mu)\text{\ s.t.\ }
      \\ f(x,y)=\frac{x}{x+y}f_1(x+y)+\frac{y}{x+y}f_2(x+y)+f_3(x,y)\Big\}, \label{eq:def-V-cont-case}
    \end{multline}
    where $\textup{Leb}$ denotes Lebesgue's measure. More precisely, for all $f\in V$, the functions
    $f_1,f_2,f_3$ in~(\ref{eq:def-V-cont-case}) are unique and there exists unique sequences
    $\{\alpha_k\}_{k\geq 1}$, $\{\beta_k\}_{k\geq 1}$, and $\{\gamma_{dk}\}_{d\geq 2,\ k\geq 1}$ such that
    \begin{multline}
      f(x,y)=\sum_{k\geq 1}\alpha_k\frac{x}{x+y}\psi_k(x+y)+\sum_{k\geq 1}\beta_k\frac{y}{x+y}\psi_k(x+y) \\
      +\sum_{d\geq 2\ k\geq 1}\gamma_{dk}H_d\Big(\frac{x-y}{x+y}\Big)\psi^{(d)}_k(x+y),
      \label{eq:basis-V}
    \end{multline}
    where the series $\sum_k\alpha_k\psi_k$ and $\sum_k\beta_k\psi_k$ both converge for
    $\|\cdot\|_{\wi{\nu}}$ and $\sum_{d,k}\gamma_{dk}H_d(\frac{x-y}{x+y})\psi^{(d)}_k(x+y)$ converges for
    $\|\cdot\|_\mu$.
  \end{description}
\end{thm}

Point~(b) says that the eigenvectors of the form~(\ref{eq:eigenv-A}), although not orthogonal in some
Hilbert space, allow one to recover a bigger class of functions than in Point~(a). 
The vector space $V$ is not equal to
$\LL^2(D,\textup{Leb})$, but the following result shows that it is much bigger than $\LL^2(D,\mu)$.

\begin{prop}
  \label{prop:V}
  The vector space $V$ of Theorem~\ref{thm:eigen-SDE} contains $H^1(D,\textup{Leb})$.
\end{prop}

\begin{proof} 
To prove this result, it is more convenient to consider the variables $(z,w)$ as in~(\ref{eq:var-z-w})
instead of $(x,y)$. The vector space $V$ then becomes the set of $g\in\LL^2([a,b]\times[-1,1],\textup{Leb})$
such that 
\begin{equation}
  \label{eq:proof-V}
  g(z,w)=\frac{1+w}{2}g_1(z)+\frac{1-w}{2}g_2(z)+g_3(z,w)  
\end{equation}
for some $g_1,g_2\in\LL^2([a,b],\textup{Leb})$ and
$g_3\in\LL^2([a,b]\times[-1,1],\nu)=\LL^2([a,b]\times[-1,1],(1-w^2)^{-1}dz\,dw)$.

We first introduce the following notion of trace: we say that a function
$g\in\LL^2([a,b]\times[-1,1],\textup{Leb})$ admits the function
$\bar{g}\in\LL^2([a,b],\textup{Leb})$ as a trace at $w=1$, or $w=-1$ respectively, if
$$
g(z,w)-\bar{g}(z)\in\LL^2([a,b]\times[0,1],(1-w^2)^{-1}dz\,dw),
$$
or
$$
g(z,w)-\bar{g}(z)\in\LL^2([a,b]\times[-1,0],(1-w^2)^{-1}dz\,dw)
$$
respectively.

Our first claim is that any $g\in\LL^2([a,b]\times[-1,1],\textup{Leb})$ which admits traces $g_1$ and
$g_2$ at $w=1$ and $w=-1$ respectively, belongs to $V$, and these traces are exactly the functions $g_1$ and
$g_2$ in~(\ref{eq:proof-V}). To see that, we only have to check that
$g_3\in\LL^2([a,b]\times[0,1],(1-w^2)^{-1}dz\,dw)$, and the same result on $[a,b]\times[-1,0]$ will follow by
symmetry:
\begin{multline*}
  \int_a^b dz\int_0^1 dw\,\frac{\left(g-\frac{1+w}{2}g_1-\frac{1-w}{2}g_2\right)^2}{1-w^2} \\
  \leq 2\int_a^bdz\int_0^1
  dw\,\frac{(g-g_1)^2}{1-w}+2\int_a^bdz\int_0^1dw\,(g_1+g_2)^2\frac{1-w}{4(1+w)}<+\infty.
\end{multline*}

Second, we claim that any $g\in H^1([a,b]\times[-1,1],\textup{Leb})$ admits traces at $w=1$ and $w=-1$ as
defined above. Assume first that $g\in C^1([a,b]\times[0,1])$. Then, using the Cauchy-Schwartz inequality,
\begin{multline*}
  \int_0^1dw\int_a^bdz\,\frac{(g(z,w)-g(z,1))^2}{1-w^2}
  =\int_0^1dw\int_a^bdz\,\frac{\left(\int_w^1\nabla_wg(z,x)dx\right)^2}{1-w^2} \\
  \leq\int_0^1dw\int_a^bdz\int_w^1dx\,|\nabla_wg(z,x)|^2\leq \|\nabla_w g\|^2_{\textup{Leb}}.
\end{multline*}
Since in addition $\|g(\cdot,1)\|_{\LL^2(\textup{Leb})}\leq 4\|g\|_{H^1(\textup{Leb})}$ by classical trace
results (cf.\ e.g.~\cite[p.\:196]{brezis}), the function $g\mapsto (g(\cdot,1),\,g-g(\cdot,1))$ extends by
density to a linear operator $\psi$ continuous from $H^1([a,b]\times[0,1],\textup{Leb})$ to
$\LL^2([a,b],\textup{Leb})\times\LL^2([a,b]\times[0,1],(1-w^2)^{-1}dz\,dw)$. Since obviously
$\psi_1(g)+\psi_2(g)=g$, the claim is proved and the proof of Proposition~\ref{prop:V} is completed.
\end{proof}

\paragraph{Proof of Theorem~\ref{thm:eigen-SDE}}
An eigenvector of $A$ of the form $H_d(w)\psi(z)$ satisfies the Neumann boundary condition in $D$ iff
$\psi'(a)=\psi'(b)=0$. The problem~(\ref{eq:EDO-vp-2}) with this boundary condition is a regular
Sturm-Liouville problem with the weight $\wi{\nu}$ defined in~(\ref{eq:def-tilde-nu}). Therefore (cf.\
e.g.~\cite[Thm.\:4.6.2]{zettl-05}), for all $d\geq 0$, there exists an orthonormal basis
$\{\psi^{(d)}_k\}_{k\geq 1}$ of $\LL^2([a,b],\wi{\nu})$ composed of solutions to~(\ref{eq:EDO-vp-2}) on
$(a,b)$ with Neumann boundary conditions. All the corresponding eigenvalues are real, simple and the
corresponding spectrum has no accumulation point.

Now, we claim that
$$
{\cal F} \df \bigcup_{d\geq 2}\Big\{H_d\Big(\frac{x-y}{x+y}\Big)\psi^{(d)}_k(x+y)\Big\}_{k\geq 1}
$$
forms an orthonormal basis of $\LL^2(D,\mu)$. The orthonormal property follows from the fact that, if
$\varphi(x,y)=H_d(\frac{x-y}{x+y})\psi^{(d)}_k(x+y)$ and
$\varphi'(x,y)=H_{d'}(\frac{x-y}{x+y})\psi^{(d')}_{k'}(x+y)$ for $d,d'\geq 2$ and $k,k'\geq 1$,
\begin{align*}
  \langle\varphi,\varphi'\rangle_\mu & =\int_{-1}^1\int_a^b
  H_d(w)H_{d'}(w)\psi^{(d)}_k(z)\psi^{(d')}_{k'}(z)d\nu(z,w) \\ &
  =\langle\psi^{(d)}_k,\psi^{(d')}_{k'}\rangle_{\wi{\nu}}\:\langle H_d,H_{d'}\rangle_{(1-w^2)^{-1}dw}.  
\end{align*}
To prove that ${\cal F}$ is a basis of $\LL^2(D,\mu)$, assume that $f\in\LL^2(D,\mu)$ satisfies
$$
\iint_{D}f(x,y)H_d\Big(\frac{x-y}{x+y}\Big)\psi^{(d)}_k(x+y)d\mu(x,y)=0,\quad\forall d\geq 2,\ k\geq 1.
$$
By the Cauchy-Schwartz inequality, for all $d\geq 2$, the function
$$
\wi{f}_d(z) \df \int_{-1}^1f\Big(\frac{z(1+w)}{2},\frac{z(1-w)}{2}\Big)\frac{H_d(w)}{1-w^2}\:dw
$$
belongs to $\LL^2([a,b],\wi{\nu})$. In addition,
$$
\langle \wi{f}_d,\psi_k^{(d)}\rangle_{\wi{\nu}}=0,\quad\forall k\geq 1.
$$
Therefore, $\wi{f}_d(z)=0$ for all $d\geq 2$, for Lebesgue-almost every $z\geq 0$. By Fubini's theorem,
$w\mapsto f(\frac{z(1+w)}{2},\frac{z(1-w)}{2})$ belongs to $\LL^2([-1,1],(1-x^2)^{-1}dx)$ for almost every
$z\geq 0$. Hence we deduce from Proposition~\ref{prop:poly-ortho} that this function is 0 for almost every
$z\geq 0$. Hence $f=0$.

Thus ${\cal F}$ is an orthonormal basis of $\LL^2(D,\mu)$ composed of eigenvectors of $A$. It is then
classical to deduce that $A$ admits no other eigenvector in this space, in the sense of point~(a).

For point~(b), let us first prove that the decomposition
$$
f=\frac{x}{x+y}f_1(x+y)+\frac{y}{x+y}f_2(x+y)+f_3(x,y)
$$
is unique for $f\in V$, with $f_1,f_2\in\LL^2([a,b],\textup{Leb})$ and $f_3\in\LL^2(D,\mu)$. We only need to
prove that this equality for $f=0$ implies $f_1=f_2=f_3=0$.

Since $f_3\in\LL^2(D,\mu)$, we have
$$
\int_0^\varepsilon dx\int_{a-x}^{b-x}dy\, f_3^2(x,y)\leq
\varepsilon\int_0^\varepsilon dx\int_{a-x}^{b-x}dy\, \frac{f_3^2(x,y)}{x}=o(\varepsilon)
$$
as $\varepsilon\rightarrow 0$. Therefore,
\begin{align*}
  \int_0^\varepsilon dx\int_{a-x}^{b-x}dy\, \frac{y^2}{(x+y)^2}f_2^2 & =
  \int_0^\varepsilon dx\int_{a-x}^{b-x}dy\, \left(\frac{x}{x+y}f_1^2+f_3\right)^2 \\ &
  \leq 2\int_0^\varepsilon dx\int_{a-x}^{b-x}dy\, \frac{x^2}{(x+y)^2}f_1^2+o(\varepsilon)=o(\varepsilon).
\end{align*}
This implies that $\int_a^b f_2^2(z)dz=0$, i.e.\ $f_2=0$. Similarly, $f_1=0$ and thus $f_3=0$.

Since $\LL^2([a,b],\textup{Leb})=\LL^2([a,b],\wi{\nu})$, the result then follows from the decomposition of
$f_1$ and $f_2$ (resp.\ $f_3$) in the orthonormal basis $\{\psi^{(1)}_k\}_{k\geq 1}$ of
$\LL^2([a,b],\wi{\nu})$ (resp.\ $\{H_d(\frac{x-y}{x+y})\psi^{(d)}_k(x+y)\}_{d\geq 2,\, k\geq 1}$ of
$\LL^2(D,\mu)$\,).\hfill$\Box$
\bigskip

To motivate the calculations of the next section, let us finally observe that, for all $\varphi\in
C^2(\RR_+^2)$,
$$
A\varphi(x,y)=\wi{T}\varphi(x,y)p(x+y)+\wi{L}\varphi(x,y)q(x+y),
$$
where
$$
\wi{L}=x\frac{\partial}{\partial x}+y\frac{\partial}{\partial y}
$$
and
\begin{equation}
  \label{eq:def-tilde-T}
  \wi{T}=x\frac{\partial^2}{\partial x^2}+y\frac{\partial^2}{\partial y^2},  
\end{equation}
and that $\wi{T}\wi{L}=\wi{L}\wi{T}$.

\section{On a family of bivariate polynomials}
\label{sec:poly}


The goal of this section is to prove the existence of a family of polynomials in $\mathbb{R}$ of two variables
$X$ and $Y$, satisfying the family of relations
\begin{subequations}
  \label{eq:poly}  
  \begin{equation}
    \label{eq:poly-1}  
    XP(X+1,Y)+YP(X,Y+1)=(X+Y+d)P(X,Y) 
  \end{equation}
  \begin{equation}
    \label{eq:poly-2} 
    XP(X-1,Y)+YP(X,Y-1)=(X+Y-d)P(X,Y)
  \end{equation}
\end{subequations}
for an integer $d\geq 0$. 

Before stating the main result of the section, let us recall some notation. $\RR[X]$ is the set of polynomials
on $\RR$ with a single variable $X$ and $\RR[X,Y]$ the set of real polynomials with two variables $X$ and
$Y$. The degree $\text{deg}(P)$ of a polynomial $P\in\RR[X,Y]$ is defined as the maximal total degree of each
monomial of $P$. We define
$$
{\cal P}_d=\{P\in\RR[X,Y]:\text{deg}(P)\leq d\}.
$$
For all $P\in\RR[X,Y]$, we may write
\begin{equation*}
  P(X,Y)=\sum_{i,j\geq 0}a_{i,j}X^iY^j,
\end{equation*}
where only finitely many of the $a_{i,j}$ are nonzero.  The real number $a_{i,j}$ will be called the
$(i,j)$-coefficient of $P$.

For any $P\in\mathbb{R}[X,Y]$ and for any $d\geq 0$, we denote by $[P]_d$ the sum of all monomials of $P$ of
degree $d$:
$$
[P]_d(X,Y)=\sum_{i=0}^da_{i,d-i}X^iY^{d-i}.
$$
In particular, $[P]_d$ is homogeneous of degree $d$ and $P=\sum_{i=1}^\infty [P]_i$.

We denote by $\Delta_i$ the first-order symmetrized discrete derivative with respect to the $i$-th variable:
\begin{align*}
  \forall P\in\mathbb{R}[X,Y],\quad & \Delta_1P(X,Y)=\frac{P(X+1,Y)-P(X-1,Y)}{2} \\
  \mbox{and}\quad & \Delta_2P(X,Y)=\frac{P(X,Y+1)-P(X,Y-1)}{2},
\end{align*}
and by $\Delta^2_i$ the symmetrized second-order discrete derivative with respect to the $i$-th variable:
\begin{align*}
  \forall P\in\mathbb{R}[X,Y],\quad &
  \Delta^2_1P(X,Y)=\frac{P(X+1,Y)+P(X-1,Y)-2P(X,Y)}{2} \\
  \mbox{and}\quad &  \Delta^2_2P(X,Y)=\frac{P(X,Y+1)+P(X,Y-1)-2P(X,Y)}{2}.
\end{align*}
Note that the superscript in the notation $\Delta^2_i$ does not correspond to the composition of the
operator $\Delta_i$ with itself.

Finally, we define the linear operators on $\RR[X,Y]$
\begin{equation*}
  L=X\Delta_1+Y\Delta_2 \quad\text{and}\quad T=X\Delta_1^2+Y\Delta_2^2.
\end{equation*}
Then, adding and substracting the equations~(\ref{eq:poly-1}) and~(\ref{eq:poly-2}), the
system~(\ref{eq:poly}) is equivalent to
\begin{align}
  L P & =dP \label{eq:(*)} \\
  T P & =0. \label{eq:(**)}
\end{align}

We are going to prove the following result
\begin{thm}
  \label{thm:poly}
  For $d=1$, the system~\eqref{eq:poly} has a two-dimensional vector space of solutions in $\mathbb{R}[X,Y]$,
  spanned by the two polynomials $P_1^{(1)} \df X$ and $P_1^{(2)} \df Y$.

  For any $d\in\{0,2,3,\ldots\}$, the system~\eqref{eq:poly} has a one-dimensional vector space of
  solutions. All these solutions, except 0, are of degree $d$.  For $d=0$, this is the vector space of
  constants, spanned by $P_0 \df 1$. When $d\geq 2$, we denote by $P_d$ the unique solution to~\eqref{eq:poly}
  with $(d-1,1)$-coefficient equal to $-2H'_d(1)=(-1)^d2\sqrt{d(d-1)(2d-1)}$, where $H_d$ is defined in
  Remark~\ref{rem:Gegenbauer}.
\end{thm}

It can be checked that the first polynomials are
\begin{equation*}
  \begin{aligned}
    & P_0=1,\quad P^{(1)}_1=X,\quad P^{(2)}_1=Y \\
    & P_2=2\sqrt{6}\ XY, \\
    & P_3=-2\sqrt{30}\ XY(X-Y), \\
    & P_4=4\sqrt{21}\ XY(X^2-3XY+Y^2+1), \\
    & P_5=-6\sqrt{20}\ XY(X-Y)(X^2-5XY+Y^2+5). 
  \end{aligned}
\end{equation*}

Before proving this result, let us give some properties of the polynomials $P_d$, proved after the proof of
Theorem~\ref{thm:poly}.
\begin{prop}
  \label{prop:Pd}
  The polynomials $P_d$, $d\geq 2$, defined in Theorem~\ref{thm:poly} satisfy the following properties:
  \begin{description}
  \item[\textmd{(a)}] For all $d\geq 2$, $[P_d]_d(X,Y)=(X+Y)^dH_d\left(\frac{X-Y}{X+Y}\right)$, where $H_d$ is
    defined in Remark~\ref{rem:Gegenbauer}.
  \item[\textmd{(b)}] $[P]_{d-2k-1}=0$ for all $0\leq k < d/2$.
  \item[\textmd{(c)}] For all $d\geq 2$, $P_d$ is divisible by $XY$. For $d$ odd, $P_d(X,Y)$ is divisible by
    $XY(X-Y)$.
  \item[\textmd{(d)}] for all $d\geq 2$, $P_d(Y,X)=P_d(-X,-Y)=(-1)^dP_d(X,Y)$.
  \item[\textmd{(e)}] $P_d(i,j)=0$ if $i,j\in\mathbb{Z}$, $ij\geq 0$ and $0\leq |i|+|j|\leq d-1$.
  \item[\textmd{(f)}] For all $d\geq 0$, the matrix $(P_i(j,d-j))_{0\leq i,j\leq d}$ is invertible, where
    $P_1=P_1^{(1)}$. In particular, $(P_d(j,d-j))_{0\leq j\leq d}\not=0$.
  \item[\textmd{(g)}] For all $d\geq 3$, $P_d(j,d-j)P_d(j+1,d-j-1)<0$ if $1\leq j\leq d-2$.
  \item[\textmd{(h)}] For all $d,d',k\geq 2$,
    \begin{equation}
      \label{eq:orth-Pd}
      \sum_{i=1}^{k-1}\frac{P_d(i,k-i)P_{d'}(i,k-i)}{i(k-i)}=2\binom{k+d-1}{2d-1}\delta_{dd'},      
    \end{equation}
    where $\delta_{ij}$ is the Kronecker symbol and by convention $\binom{i}{j}=0$ if $j<0$ or $j>i$.
  \end{description}
\end{prop}

\begin{prop}
  \label{prop:explicit-Pd}
  For all $d\geq 2$, the polynomial $P_d$ is given by the following formula:
  \begin{equation}
    \label{eq:explicit-Pd}
    P_d(X,Y)=C_d(-X-Y)_d\sum_{k=1}^d\frac{(-d)_k(d-1)k}{(k-1)!k!}\:\frac{(-X)_k}{(-X-Y)_k},    
  \end{equation}
  where
  $$
  C_d=(-1)^{d+1}2\sqrt{\frac{2d-1}{d(d-1)}}.
  $$
\end{prop}

In order to prove Theorem~\ref{thm:poly}, we need the following lemma.
\begin{lembr}
  \label{lem:poly}
  \begin{description}
  \item[\textmd{(a)}] We have
    \begin{equation}
      \label{eq:TL=LT+T}
      TL=LT+T,
    \end{equation}
  \item[\textmd{(b)}] Define for all $d\geq 0$
    $$
    \cD_d=\{P\in\cP_d\,:\, T(P)=0\}.
    $$
    Then $\text{dim}({\cal D}_0)=1$ and $\text{dim}({\cal D}_d)=d+2$ for all $d\geq 1$.
  \end{description}
\end{lembr}

\paragraph{Proof of Lemma~\ref{lem:poly}}
To prove~(\ref{eq:TL=LT+T}), it would be enough to expand both sides of the equation. We prefer to give a
proof based on differential calculations, because it is related to the method used in the rest of the proof.
First, let $I=\{1,3,5,\ldots\}$ and $J=\{2,4,6,\ldots\}$ be the sets of odd and even integers, respectively.
Using the fact that, for $Q\in\RR[X]$.
$$
Q(X+1)=\sum_{n\geq 0}\frac{Q^{(n)}(X)}{n!}\quad\mbox{and}\quad Q(X-1)=\sum_{n\geq
  0}\frac{(-1)^nQ^{(n)}(X)}{n!},
$$
where $Q^{(n)}$ denotes the $n$-th derivative of $Q$, one has
\begin{align}
  T & =\sum_{p\in J}\frac{1}{p!}\left(X\frac{\partial^p}{\partial X^p}+Y\frac{\partial^p}{\partial
      Y^p}\right), \label{eq:T} \\ L & =\sum_{q\in I}\frac{1}{q!}\left(X\frac{\partial^q}{\partial
      X^q}+Y\frac{\partial^q}{\partial Y^q}\right). \label{eq:L}
\end{align}
Since for all $p,q\in\NN$
$$
X\frac{\pa^p}{\pa X^p}\left(X\frac{\pa^q}{\pa X^q}\right)=X^2\frac{\pa^{p+q}}{\pa
  X^{p+q}}+pX\frac{\pa^{p+q-1}}{\pa X^{p+q-1}},
$$
one easily checks that
$$
TL-LT=\sum_{p\in J,q\in I}\frac1{p!}\frac1{q!}(p-q)\left(X\frac{\partial^{p+q-1}}{\partial
  X^{p+q-1}}+Y\frac{\partial^{p+q-1}}{\partial Y^{p+q-1}}\right).
$$
Now, for all $n\in\NN$,
\begin{align*}
  \sum_{p\in J,\,q\in I,\, p+q=2n+1}\frac1{p!}\frac1{q!}(p-q) &
  =\sum_{p\in I,\,q\in I,\, p+q=2n}\frac1{p!}\frac1{q!}-\sum_{p\in J,\,q\in J\cup\{0\},\,
    p+q=2n}\frac1{p!}\frac1{q!}\\
  & =-\sum_{0\leq p\leq 2n}\frac1{p!}\frac1{(2n-p)!}(-1)^p+\frac1{(2n)!} \\
  & =\lt((1-1)^{2n}+1\rt)\frac1{(2n)!}=\frac1{(2n)!}.
\end{align*}
This completes the proof of Lemma~\ref{lem:poly}~(a).

To prove~(b), let ${\cal H}_p$ be the subspace of ${\cal P}_d$ composed of all homogeneous polynomials of
degree $d$. Recall the definition~(\ref{eq:def-tilde-T}) of the operator $\wi{T}$ on $\RR[X,Y]$, and
observe that, for all $d\geq 1$, $\wi{T}$ is a linear map from ${\cal H}_d$ to ${\cal H}_{d-1}$. Now, the
family $\{(X-Y)^k(X+Y)^{d-k},0\leq k\leq d\}$ forms a basis of ${\cal H}_d$. Hence, any $P\in{\cal H}_d$ can
be written in the form $P(X,Y)=(X+Y)^dh\left(\frac{X-Y}{X+Y}\right)$, where $h\in\RR[X]$ has degree $d$. With
this notation, it can be checked that
\begin{multline*}
  \frac{\pa^2 P}{\pa X^2}(X,Y)=d(d-1)(X+Y)^{d-2}h(W) \\
  +4(d-1)Y(X+Y)^{d-3}h'(W)+4Y^2(X+Y)^{d-4}h''(W),  
\end{multline*}
where $W=(X-Y)/(X+Y)$, and similarly for the second variable. This yields
$$
\wi{T}P=(X+Y)^{d-1}\left(d(d-1)h(W)+4\frac{XY}{(X+Y)^2}h''(W)\right).
$$
Using the relation $4XY/(X+Y)^2=1-W^2$, we finally obtain that $P\in\text{Ker}(\wi{T})\cap{\cal H}_d$ if
and only if $h$ solves~(\ref{eq:ODE-Hd}). By Proposition~\ref{prop:poly-ortho}, for all $d\not=1$, this
equation has a unique (up to a multliplicative constant) polynomial solution, which has degree $d$. Since
$\text{dim}({\cal H}_d)=d+1$, we deduce that $\wi{T}:{\cal H}_d\rightarrow{\cal H}_{d-1}$ is surjective for
all $d\geq 2$.  If $d=1$, $\text{Ker}(\wi{T})\cap{\cal H}_1={\cal H}_1$ which has dimension 2.

Now, let $P=[P]_1+\ldots+[P]_d\in{\cal P}_d$ and observe that any $k$-th order derivative of $[P]_m$ belongs
to ${\cal H}_{m-k}$ if $k\leq m$. Therefore, by~(\ref{eq:T}), the equation $TP=0$ is equivalent to the fact
that, for all $0\leq n\leq d-1$,
$$
[TP]_n=\sum_{p\geq 1} \frac1{(2p)!}\left(X\frac{\pa^{2p}}{\pa X^{2p}}
  +Y\frac{\pa^{2p}}{\pa Y^{2p}}\right)[P]_{n+2p-1}=0,
$$
or, equivalently,
\begin{equation}
  \label{eq:T-tilde}
  \wi T [P]_{n+1}=-2\sum_{p\geq 2} \frac1{(2p)!}\left(X\frac{\pa^{2p}}{\pa X^{2p}}
    +Y\frac{\pa^{2p}}{\pa Y^{2p}}\right)[P]_{n+2p-1}.
\end{equation}
If $n\geq 1$ and $[P]_{n+3},[P]_{n+5},\ldots$ are given, there is a one-dimensional affine space of solution
for this equation. If $n=0$,~(\ref{eq:T-tilde}) is automatically satisfied, since both sides are 0. Therefore,
choosing recursively $[P]_d,[P]_{d-1},\ldots,[P]_2$ and setting any value to $[P]_1$ and $[P]_0$, the result
on the dimension of ${\cal D}_d$ easily follows.\hfill$\Box$\bigskip

\paragraph{Proof of Theorem~\ref{thm:poly}}
Fix $d\geq 0$. We claim that, as a linear operator on ${\cal P}_d$, $L$ is diagonalizable and its spectrum
$\text{Sp}_{{\cal P}_d}(L)=\{0,1,\ldots,d\}$. To see this, fix $\lambda\in\text{Sp}_{{\cal P}_d}(L)$ and $P$
an eigenvector for this eigenvalue, with degree $p$. Writing as in the proof of Lemma~\ref{lem:poly}
$P=[P]_p+\ldots+[P]_0$, the equation $LP=\lambda P$ is equivalent to the fact that, for $0\leq n\leq p$
\begin{equation}
  \label{eq:P-tilde}
  \lambda[P]_n=\sum_{q\geq 0}\frac{1}{(2q+1)!}\left(X\frac{\pa^{2q+1}}{\pa X^{2q+1}}
    +Y\frac{\pa^{2q+1}}{\pa Y^{2q+1}}\right)[P]_{n+2q}.
\end{equation}
Now, for any $Q\in{\cal H}_k$, one has
\begin{equation}
  \label{eq:homogene}
  \wi{L}Q=X\frac{\pa Q}{\pa X}+Y\frac{\pa Q}{\pa Y}=k Q.  
\end{equation}
Therefore,~(\ref{eq:P-tilde}) for $n=p$ imposes $\lambda=p$, and for $n=p-1$, $[P]_{p-1}=0$. Moreover, for
$0\leq n\leq p-2$,~(\ref{eq:P-tilde}) is equivalent to
\begin{equation}
  \label{eq:P-tilde-bis} 
  (p-n)[P]_n=\sum_{q\geq 1}\frac{1}{(2q+1)!}\left(X\frac{\pa^{2q+1}}{\pa X^{2q+1}}
    +Y\frac{\pa^{2q+1}}{\pa Y^{2q+1}}\right)[P]_{n+2q},
\end{equation}
which allows one to compute recursively $[P]_{p-2},\ldots,[P]_0$ given any $[P]_p\in{\cal H}_p$ and
$[P]_{p-1}=0$. Since $\text{dim}({\cal H}_p)=p+1$, the eigenspace corresponding to the eigenvalue $p$ of $L$
has dimension $p+1$.

Now, it follows from~(\ref{eq:TL=LT+T}) that $L$ is a linear operator from ${\cal D}_d$ to ${\cal D}_d$. Since
$L$ is diagonalizable in ${\cal P}_d$, it is also diagonalizable on any stable subspace, and $\text{Sp}_{{\cal
    D}_d}(L)\subset\{0,\ldots,d\}$.

Let $p\in\{2,\ldots,d\}$ and assume that there exists $P\in{\cal D}_d\setminus\{0\}$ satisfying
$LP=pP$. Again, $\text{deg}(P)=p$ necessarily. Writing $P=[P]_p+\ldots+[P]_0$ again, since $T(P)=0$, we have
$$
\wi{T}[P]_p=X\frac{\pa^2[P]_p}{\pa X^2}+Y\frac{\pa^2[P]_p}{\pa Y^2}=0,
$$
which has a one-dimensional vector space of solutions in ${\cal H}_p$. Once $[P]_p$ is fixed and since we have
$[P]_{p-1}=0$,~(\ref{eq:P-tilde-bis}) can be used recursively to compute $[P]_{p-2},\ldots,[P]_0$. In
conclusion, the eigenspace of $L$ in ${\cal D}_d$ for the eigenvalue $p$ is either of dimension 1 or 0. Now,
$L$ is diagonalizable in ${\cal D}_d$. Since $\text{dim}({\cal D}_d)=d+2$ and $\text{dim}({\cal
  H}_0)+\text{dim}({\cal H}_1)=3$, the only possibility is that $\text{Sp}_{{\cal D}_d}(L)=\{0,1,\ldots,d\}$
and that each eigenvalue $p\not=1$ has a one-dimensional vector space of solutions, and the eigenvalue $1$ has
a two-dimensional vector space of solutions.

This easily implies Theorem~\ref{thm:poly}, except for the expression of $-2H'_d(1)$. This
can be easily obtained from~(\ref{eq:Gegen}) and the relation $H'_d(1)=(-1)^{d-1}H'_d(-1)$
which follows from the parity property of $H_d$ stated in Proposition~\ref{prop:poly-ortho}.
\hfill$\Box$\bigskip

\paragraph{Proof of Proposition~\ref{prop:Pd}}
Recall from the proof of Lemma~\ref{lem:poly} that a nonzero solution $P$ of~(\ref{eq:poly})
for $d\geq 2$ satisfies
$$
[P]_d(X,Y)=\sum_{n=0}^d b_n(X-Y)^n(X+Y)^{d-n},
$$
where $h(x)=\sum_{n=0}^d b_n x^n$ is a polynomial solution of~(\ref{eq:ODE-Hd}). Therefore, the
$(d-1,1)$-coefficient $a_{d-1,1}$ of $P$ is given by
$$
a_{d-1,1}=\sum_{n=0}^d b_n(-n+d-n)=dh(1)-2h'(1)=-2h'(1),
$$
and Point~(a) then follows from Proposition~\ref{prop:poly-ortho} and the value of the
$(d-1,1)$ coefficient of $P_d$.

Observe that any polynomial solution of~(\ref{eq:ODE-Hd}) with $d\geq
2$ is divisible by $(X-1)(X+1)$. As a consequence of the previous construction, any polynomial $P$ such
that~(\ref{eq:poly}) holds satisfies that $XY$ divides $[P]_d$. Note also that $[P]_{d-1}=0$, which
implies~(b) by~(\ref{eq:P-tilde-bis}). Moreover,~(\ref{eq:P-tilde-bis}) also implies by induction that $XY$
divides $[P]_{d-2k}$ for all $0\leq k\leq d/2$, which yields the first part of~(c).

By Proposition~\ref{prop:poly-ortho},~(d) is true for $[P_d]_d$ and of course for $[P_d]_{d-1}=0$. Now, assume
that $P\in\RR[X,Y]$ satisfies $P(Y,X)=P(-X,-Y)=(-1)^\alpha P(X,Y)$. Then it can be easily checked that, for
all $k\geq 1$,
$$
Q(X,Y) \df X\frac{\pa^k P}{\pa X^k}(X,Y)+Y\frac{\pa^k P}{\pa Y^k}(X,Y)
$$
satisfies
$$
Q(Y,X)=(-1)^\alpha Q(X,Y) \quad\mbox{and}\quad Q(-X,-Y)=(-1)^{\alpha+k+1}Q(X,Y).
$$
Therefore,~(d) easily follows from~(\ref{eq:P-tilde-bis}) by induction.

Now, fix $d$ odd. By Proposition~\ref{prop:poly-ortho}, $H_d$ is odd and thus $H_d(X)$ is divisible by
$X$. This implies that $[P]_d$ is divisible by $X-Y$. Moreover, it follows from~(d) that the polynomial
$[P]_d/(X-Y)$ is symmetric. Now, let $Q(X,Y)=X^nY^m+X^mY^n$ for some $n,m\geq 0$, and fix $k\geq 0$. Since
\begin{align*}
  X\frac{\pa^k Q}{\pa X^k}-Y\frac{\pa^k Q}{\pa Y^k} & =n(n-1)\ldots(n-k+1)(X^{n-k+1}Y^m-X^mY^{n-k+1}) \\ &
  +m(m-1)\ldots(m-k+1)(X^{m-k+1}Y^n-X^nY^{m-k+1}),
\end{align*}
the polynomial $X-Y$ divides $X\frac{\pa^k Q}{\pa X^k}-Y\frac{\pa^k Q}{\pa Y^k}$. Since this holds for any
$n,m\geq 0$, the same is true for all $Q$ such that $Q(X,Y)=Q(Y,X)$. Now, any polynomial of the form
$P(X,Y)=(X-Y)Q(X,Y)$ with $Q$ symmetric satisfies
\begin{multline*}
  X\frac{\pa^k P}{\pa X^k}(X,Y)+Y\frac{\pa^k P}{\pa Y^k}(X,Y)
  =kX\frac{\pa^{k-1} Q}{\pa X^{k-1}}(X,Y)-kY\frac{\pa^{k-1} Q}{\pa Y^{k-1}}(X,Y) \\
  +(X-Y)\left(X\frac{\pa^k Q}{\pa X^k}(X,Y)+Y\frac{\pa^k Q}{\pa Y^k}(X,Y)\right).
\end{multline*}
In particular, $X-Y$ divides $X\frac{\pa^k P}{\pa X^k}+Y\frac{\pa^k P}{\pa Y^k}$. Therefore, the fact that
$[P]_i$ is divisible by $X-Y$ for $i<d$ follows form~(\ref{eq:P-tilde-bis}) by induction. This ends the proof
of~(c).

As a consequence of~(c), $P_d(i,0)=P_d(0,j)=0$ for any $i,j\in\mathbb{Z}$ for $d\geq
2$. Applying~(\ref{eq:poly-2}) for $(X,Y)=(d-1,1)$ yields $P_d(d-2,1)=0$. By induction,
applying~(\ref{eq:poly-2}) for $(X,Y)=(d-k,1)$ implies $P_d(d-k-1,1)=0$ for all
$k\in\{1,\ldots,d-2\}$. Similarly, applying~(\ref{eq:poly-2}) for $(X,Y)=(d-1-k,2)$ implies $P_d(d-2-k,2)=0$
for all $k\in\{1,\ldots,d-3\}$. Point~(e) is therefore straightforward by induction.

For all $d\leq k$, the polynomial $Q_{k,d}(X)=P_k(X,d-X)$ satisfies $[Q_{k,d}]_k=\left[d^k
  H_k\left(\frac{2X-d}{d}\right)\right]_k\not=0$ for all $k\geq 2$. Therefore, $\text{deg}(Q_{k,d})=k$ for all
$k\geq 0$, and $\{Q_{0,d},Q_{1,d},\ldots,Q_{d,d}\}$ is a basis of ${\cal P}'_d \df \{Q\in\RR[X]:\text{deg}(Q)\leq
d\}$. Since $\varphi(Q)=(Q(0),\ldots,Q(d))$ defines a linear isomorphism from ${\cal P}'_d$ to $\RR^{d+1}$, we
deduce that $\{\varphi(Q_{0,d}),\ldots,\varphi(Q_{d,d})\}$ is a basis of $\RR^{d+1}$, which is equivalent
to~(f).


Point~(g) is a simple consequence of points~(e) and~(f) and of formula~(\ref{eq:poly-1}) with $X=j$ and
$Y=d-j-1$.

Because of point~(e) above,~(h) is obvious if $k\leq d-1$ or $k\leq d'-1$. So let us assume that $d,d'\leq
k$. Multiplying~(\ref{eq:poly-1}) by $(X+Y-d+1)$ and applying~(\ref{eq:poly-2}) to both terms on the l.h.s.\
yields
\begin{equation}
  \label{eq:utile-Moran-3}
  (2XY-d(d-1))P_d(X,Y)=XY(P_d(X+1,Y-1)+P_d(X-1,Y+1)).  
\end{equation}
This means that, for all $k\geq 2$ and $2\leq d\leq k$, the vector $(P_d(i,k-i))_{1\leq i\leq k-1}$ is a right
eigenvector of the matrix $A_k=(a^{(k)}_{i,j})_{1\leq i,j\leq k-1}$ for the eigenvalue $-d(d-1)$, where
\begin{align*}
  a^{(k)}_{i,i} =-2i(k-i) & \text{\ for\ }1\leq i\leq k-1, \\
  a^{(k)}_{i,i+1} =i(k-i) & \text{\ for\ }2\leq i\leq k-1, \\
  a^{(k)}_{i,i-1} =i(k-i) & \text{\ for\ }1\leq i\leq k-2 \\
  \text{and\ }a^{(k)}_{ij} =0 & \text{\ for\ }|i-j|\geq 2.
\end{align*}
It is straightforward to check that the matrix $A_k$ is self-adjoint for the inner product
$\langle\cdot,\cdot\rangle_\mu$, where $\mu_i=1/i(k-i)$, which implies that two right eigenvectors of $A_k$
corresponding to different eigenvalues are orthogonal w.r.t.\ this inner product. This
yields~(\ref{eq:orth-Pd}) for $d\not=d'$.

Finally, fix $2\leq d\leq k$. Using~(\ref{eq:poly-1}) and~(\ref{eq:poly-2}), we have
\begin{multline*}
  (k+d)\sum_{i=1}^{k-1}\frac{P_d(i,k-i)^2}{i(k-i)} \\
  \begin{aligned}
    & =\sum_{i=1}^{k-1}P_d(i,k-i)\frac{P_{d}(i+1,k-i)}{k-i}+\sum_{i=1}^{k-1}P_d(i,k-i)\frac{P_{d}(i,k-i+1)}{i} \\
    & =\sum_{i=1}^{k}P_d(i,k-i+1)\frac{P_d(i-1,k-i+1)}{k-i+1}+\sum_{i=1}^{k}P_d(i,k-i+1)\frac{P_d(i,k-i)}{i} \\
    & =(k-d+1)\sum_{i=1}^{k}\frac{P_d(i,k-i+1)^2}{i(k-i+1)}.
  \end{aligned}
\end{multline*}
Applying this equality inductively, we deduce that
$$
\sum_{i=1}^{k-1}\frac{P_d(i,k-i)^2}{i(k-i)}=C\binom{k+d-1}{2d-1}
$$
for some constant $C$.

Now, by point~(a) and Proposition~\ref{prop:poly-ortho}, we have
$$
\frac{1}{k^{2d-1}}\sum_{i=1}^{k-1}\frac{P^2_d(i,k-i)}{i(k-i)}
\sim\frac{4}{k}\:\sum_{i=1}^{k-1}\frac{H_d^2\left(\frac{2i-k}{k}\right)}
{1-\left(\frac{2i-k}{k}\right)^2}\longrightarrow 2\int_{-1}^1\frac{H^2_d(x)}{1-x^2}dx=2
$$
as $k\rightarrow+\infty$. Thus $C=2$ and the proof of~(h) is completed.\hfill$\Box$\bigskip

\paragraph{Proof of Proposition~\ref{prop:explicit-Pd}}

In~\cite{karlin-mcgregor-75}, the authors construct a family of functions of two variables
satisfying relations close to~(\ref{eq:poly}), which they use to study neutral, multitype
population processes with non-zero mutation or immigration. These functions are expressed in
terms of the Hahn polynomials, defined for fixed parameters $\alpha>-1$, $\beta>-1$ and
$N\in\mathbb{N}$ by
\begin{equation}
  \label{eq:Hahn-poly}
  Q_d(x;\alpha,\beta,N)={_3}F_2\biggl(\begin{matrix}-d,\ -x,\ d+\alpha+\beta+1\\ \alpha+1,\ -N+1\end{matrix};1\biggr),
\end{equation}
for all integer $d\geq 0$. Karlin and McGregor proved that the rational function
$$
\phi_d(X,Y)=Q_d(X;\alpha,\beta,X+Y+1)=\sum_{k=0}^d\frac{(-d)_k(-X)_k(d+\alpha+\beta+1)_k}{(\alpha+1)_k(-X-Y)_kk!}
$$
satisfies
\begin{gather}
  \label{eq:Karlin-McGregor-1}
  (X+\alpha+1)\phi_d(X+1,Y)+(Y+\beta+1)\phi_d(X,Y+1)=(X+Y)\phi_d(X,Y)
  \\  \label{eq:Karlin-McGregor-2}
  \begin{aligned}
    & X\phi_d(X-1,Y)+Y\phi_d(X,Y-1) \\ & \qquad\qquad=\frac{(X+Y+1-d)(X+Y+d+\alpha+\beta+2)}{X+Y+1}\phi_d(X,Y).
  \end{aligned}
\end{gather}
Let us define
\begin{align*}
  \psi_d(X,Y) &
  =(-X-Y)_d\lim_{\beta\rightarrow-1}\lim_{\alpha\rightarrow-1}(\alpha+1)\phi_d(X,Y) \\
  & =(-X-Y)_d\sum_{k=1}^d\frac{(-d)_k(-X)_k(d-1)_k}{(k-1)!(-X-Y)_kk!}.
\end{align*}
Passing to the limit in~(\ref{eq:Karlin-McGregor-1}) and~(\ref{eq:Karlin-McGregor-2}) proves
that $\psi_d$ satisfies~(\ref{eq:poly}). Since $\psi_d$ is a polynomial,
Theorem~\ref{thm:poly} entails~(\ref{eq:explicit-Pd}). It only remains to check that
$C_d\psi_d(X,Y)$ has its $(d-1,1)$ coefficient equal to
$(-1)^d2\sqrt{d(d-1)(2d-1)}$. The $(d-1,1)$ coefficient of
$\psi_d(X,Y)$ is
\begin{align*}
  (-1)^d\sum_{k=1}^d\frac{(-d)_k(d-1)_k}{(k-1)!k!}(d-k)
   & =(-1)^{d+1}d(d-1)^2\sum_{k=1}^{d}\frac{(-d+2)_{k-1}(d)_{k-1}}{(k-1)!k!} \\ & =(-1)^{d+1} d(d-1)
  {_2}F_1\biggl(\begin{matrix}-d+2,\ d\\ 2\end{matrix};1\biggr).
\end{align*}
The Chu-Vandermonde formula (cf.\ e.g.~\cite{dunkl-xu-01}) implies that
$$
{_2}F_1\biggl(\begin{matrix}-d+2,\ d\\ 2\end{matrix};1\biggr)=(-1)^d\frac{(d-2)!}{(d-1)!}=\frac{(-1)^d}{d-1},
$$
which gives the expression of $C_d$.\hfill$\Box$

\section{Spectral decomposition of neutral two-dimensional Markov chains}
\label{sec:diago}

In this section, we consider neutral extensions of the two-dimensional birth and death chains in $\ZZ^2$
described in the introduction. In the sequel, this family will be called N2dMC, for neutral two-dimensional
Markov chains.

A N2dMC $(X_t,Y_t)_{t\in\ZZ_+}$ is constructed by specifying first the Markov dynamics of $Z_t=X_t+Y_t$ in
$\ZZ_+$. Assume that its transition matrix is
$$
\Pi_0=(p_{n,m})_{n,m\geq 0},
$$
where $\sum_{m\geq 0}p_{n,m}=1$ for all $n\geq 1$ and the state $0$ is absorbing ($p_{0,0}=1$). Then, the
process $(X_t,Y_t)_{t\in\ZZ_+}$ is constructed as follows: if there is a birth at time $t$ (i.e.\ if
$Z_{t+1}>Z_t$), the types of the new individuals are \emph{successively} picked at random in the population;
if there is a death at time $t$ (i.e.\ if $Z_{t+1}<Z_t$), the types of the killed individuals are
\emph{successively} picked at random in the population; finally, if $Z_{t+1}=Z_t$, then $X_{t+1}=X_t$ and
$Y_{t+1}=Y_t$.

For example, the transition probability from $(i,j)$ to $(i+k,j+l)$ for $(i,j)\in\ZZ_+^2$ and $k,l\geq 0$,
$k+l\geq 1$ is
$$
\binom{k+l}{k}\frac{i(i+1)\ldots(i+k-1)\:j(j+1)\ldots(j+l-1)}{(i+j)(i+j+1)\ldots(i+j+k+l-1)}\:p_{i+j,i+j+k+l}.
$$
After some algebra, one gets the following formulas for the transition probabilities: for all $l\geq 0$ and
$k\geq 0$ such that $l+k\geq 1$, the Markov chain $(X_n,Y_n)_{n\geq 0}$ has transitions from $(i,j)\in\ZZ_+^2$
to
\begin{equation}
  \label{eq:neutral-trans}
  \begin{array}{rrcl}
    (i+k,j+l) \mbox{\ w.\ p.} & \pi_{(i,j),(i+k,j+l)} &  \df  &  \displaystyle{
      \frac{\binom{i+k-1}{k}\binom{j+l-1}{l}}{\binom{i+j+k+l-1}{k+l}}\:p_{i+j,\:i+j+k+l}} \\
    (i-k,j-l) \mbox{\ w.\ p.} & \pi_{(i,j),(i-k,j-l)} &  \df  & \displaystyle{
      \frac{\binom{i}{k}\binom{j}{l}}{\binom{i+j}{k+l}}\:p_{i+j,\:i+j-k-l}} \\
    (i,j) \mbox{\ w.\ p.} & \pi_{(i,j),(i,j)} &  \df  & p_{i+j,\:i+j},
  \end{array}
\end{equation}
with the convention that $\binom{i}{j}=0$ if $i<0$, $j<0$ or $j>i$. In particular, once one component of the
process is 0, it stays zero forever. We denote by
$$
\Pi \df (\pi_{(i,j),(k,l)})_{(i,j),(k,l)\in\ZZ_+^2}
$$
the transition matrix of the Markov process $(X,Y)$.

The state space of the Markov chain $Z$ will be denoted by ${\cal S}_Z$, and the state space of $(X,Y)$ by
${\cal S}$. We are going to consider two cases: the case where $Z$ has finite state space ${\cal
  S}_Z \df \{0,1,\ldots,N\}$ for some $N\geq 0$, and the case where $Z$ has infinite state space ${\cal
  S}_Z \df \ZZ_+$. In the first case, the state space of $(X,Y)$ is the set
\begin{equation}
  \label{eq:def-T_N}
  {\cal T}_N \df \{(i,j)\in\ZZ^2_+:i+j\leq N\}.
\end{equation}
In the second case, ${\cal S} \df \ZZ_+^2$. 

We also define the sets ${\cal S}^* \df {\cal S}\cap\NN^2$, ${\cal T}_N^* \df {\cal T}_N\cap\NN^2$ and ${\cal
  S}_Z^* \df {\cal S}_Z\setminus\{0\}$. Finally, let $\wi{\Pi}_0$ be the restriction of the matrix $\Pi_0$ to
${\cal S}_Z^*$ (i.e.\ the matrix obtained from $\Pi_0$ by suppressing the row and column of index $0$) and let
$\wi{\Pi}$ be the restriction of the matrix $\Pi$ to ${\cal S}^*$. 

Extending the usual definition for Markov chains, we say that a matrix $M=(m_{ij})_{i,j\in{\cal A}}$ is
reversible with respect to the measure $(\mu_i)_{i\in{\cal A}}$ if $\mu_i>0$ and $\mu_im_{ij}=\mu_jm_{ji}$ for
all $i,j\in{\cal A}$.

For all $d\geq 0$, we define
$$
{\cal V}_d:=\{v\in\RR^{\cal S}:v_{i,j}=P_d(i,j)u_{i+j}\text{\ with\ }u\in\RR^{{\cal S}_Z}\},
$$
where we recall that the polynomials $P_d$ are defined in Theorem~\ref{thm:poly} and
$P_1=P_1^{(1)}$. Note that, by Proposition~\ref{prop:Pd}~(e), a vector
$v_{i,j}=P_d(i,j)u_{i+j}\in {\cal V}_d$ is characterized by the values of $u_k$ for $k\geq d$
only.

For all $d\geq 0$, we also define the matrix $\Pi_d \df (p^{(d)}_{n,m})_{(n,m)\in {\cal
    S},\:n\geq d,\:m\geq d}$, where for all $(n,m)\in {\cal S}$ such that $n\geq d$ and $m\geq
d$,
\begin{equation}
  \label{eq:def-M_d}
  p^{(d)}_{n,m} \df 
  \begin{cases}
    \displaystyle{\frac{\binom{m+d-1}{m-n}}{\binom{m-1}{m-n}}\:p_{n,m}} & \text{if\ }m>n, \\
    \displaystyle{\frac{\binom{n-d}{n-m}}{\binom{n}{n-m}}\:p_{n,m}} & \text{if\ }m<n, \\
    p_{n,n} & \text{if\ }m=n.
  \end{cases}
\end{equation}

All these notation, as well as those introduced in the rest of the paper, are gathered for convenience in
Appendix~\ref{sec:nota}.
\medskip

The following result is the basis of all results in this section.

\begin{prop}
  \label{prop:stabilite}
  For all $d\geq 0$, the vector space ${\cal V}_d$ is stable for the matrix $\Pi$. In
  addition, for all $v_{i,j}=P_d(i,j)u_{i+j}\in{\cal V}_d$,
  $$
  (\Pi v)_{i,j}=P_d(i,j)(\Pi_d u)_{i+j}.
  $$
\end{prop}

\begin{proof}
%
Using~(\ref{eq:poly-1}) inductively, we have
\begin{multline*}
  \sum_{k=0}^n\binom{n}{k}X(X+1)\ldots(X+k-1)Y(Y+1)\ldots(Y+n-k-1)P_d(X+k,Y+n-k)
  \\ =(X+Y+d)(X+Y+d+1)\ldots(X+Y+d+n-1)P_d(X,Y)
\end{multline*}
for all $d\geq 0$ and $n\geq 1$, which can be written as
\begin{multline*}
  \sum_{k=0}^n\binom{X+k-1}{k}\binom{Y+n-k-1}{n-k}P_d(X+k,Y+n-k) \\ =\binom{X+Y+d+n-1}{n}P_d(X,Y)
\end{multline*}
for $(X,Y)\in\ZZ_+^2$. Similarly, an inductive use of~(\ref{eq:poly-2}) yields
$$
\sum_{k=0}^n\binom{X}{k}\binom{Y}{n-k}P_d(X-k,Y-n+k)=\binom{X+Y-d}{n}P_d(X,Y)
$$
for all $d\geq 0$, $n\geq 1$ and $(X,Y)\in\ZZ_+^2$, where we recall the convention $\binom{a}{b}=0$ if $a<0$,
$b<0$ or $b>a$.

Proposition~\ref{prop:stabilite} then easily follows from these equations and from the transition probabilities~(\ref{eq:neutral-trans}).
\end{proof}

\subsection{The case of finite state space}
\label{sec:finite}

\subsubsection{Eigenvectors of $\Pi$ for finite state spaces}
\label{sec:result-finite}

In the case where $Z$ has finite state space, the main result of this section is the following.

\begin{thm}
  \label{thm:diago-finite}
  Assume that ${\cal S}_Z={\cal T}_N$ for some $N\geq 0$.
  \begin{description}
  \item[\textmd{(a)}] For all $d\geq 0$ and all right eigenvector $(u_n)_{n\in {\cal S}_Z,\:n\geq d}$ of $\Pi_d$ for the
    eigenvalue $\theta$, the vector
    \begin{equation}
      \label{eq:eigenvector-v}
      v_{(i,j)}=P_d(i,j)u_{i+j},\quad (i,j)\in {\cal S}    
    \end{equation}
    is a right eigenvector for the eigenvalue $\theta$ of the matrix $\Pi$, where the polynomials $P_d$ are
    defined in Theorem~\ref{thm:poly} and where $P_1=P^{(1)}_1$.\\
    In addition, if $d\geq 2$, $v$ is also a right eigenvector for the eigenvalue $\theta$ of the matrix
    $\wi{\Pi}$.
  \item[\textmd{(b)}] All the right eigenvectors of $\Pi$ are of the
    form~(\ref{eq:eigenvector-v}), or possibly a linear combination of such eigenvectors in the case of
    multiple eigenvalues.
  \item[\textmd{(c)}] Assume that $\wi{\Pi}_0$ admits a positive reversible measure $(\mu_n)_{n\in {\cal
        S}^*_Z}$. Then, the matrix $\wi{\Pi}$ is reversible w.r.t.\ the measure
    \begin{equation}
      \label{eq:def-nu}
      \nu_{(i,j)} \df \frac{(i+j)\mu_{i+j}}{ij},\quad\forall(i,j)\in {\cal S}^*,
    \end{equation}
    and hence is diagonalizable in a basis orthonormal for this measure, composed of vectors of the
    form~(\ref{eq:eigenvector-v}) for $d\geq 2$.\\ In addition, $\Pi$ is diagonalizable in a basis of
    eigenvectors of the form~(\ref{eq:eigenvector-v}) for $d\geq 0$.
  \end{description}
\end{thm}

Hence, the right eigenvectors of the transition matrix of a finite N2dMC can be decomposed as the product of
two terms, one depending on each population size, but ``universal'' in the sense that it does not depend on
the transitions matrix $\Pi_0$ of $Z$, and the other depending on the matrix $\Pi_0$, but depending only on
the total population size.

\begin{rem}
  \label{rem:P_1}
  There is some redundancy among the right eigenvectors of $\Pi$ of the form $P(i,j)u_{i+j}$ for $P=P_0$,
  $P_1^{(1)}$ or $P_1^{(2)}$: if $u$ is a right eigenvector of $\Pi_1$, the vectors
  $$
  (iu_{i+j})_{(i,j)\in{\cal S}}\quad \text{and}\quad (ju_{i+j})_{(i,j)\in{\cal S}}
  $$
  are right eigenvectors of $\Pi$ for the same eigenvalue. In particular, $iu_{i+j}+ju_{i+j}$ is an
  eigenvector of $\Pi$ of the form $P_0(i,j)u'_{i+j}$.
  This will also be true when $Z$ has infinite state space.
\end{rem}

\begin{rem}
  \label{rem:continuous-time}
  In the following proof (and also in the case of infinite state space), no specific use is made of the fact
  that the matrix $\Pi$ is stochastic. Therefore, Theorem~\ref{thm:diago-finite} also holds true in the case
  of a \textbf{continuous-time} N2dMC, where the matrix $\Pi_0$ is now the infinitesimal generator of the
  process $Z_t=X_t+Y_t$.
\end{rem}

\begin{proof}
Point~(a) is an easy consequence of Proposition~\ref{prop:stabilite}.



For all $0\leq d\leq N$, the matrix $\Pi_d$ is conjugate with its Jordan normal form. Let
$\{u^{(d),k}\}_{d\leq k\leq N}$ denote the basis of $\CC^{N-d+1}$ corresponding to this normal form, where
$u^{(d),k}=(u^{(d),k}_d,u^{(d),k}_{d+1}\ldots,u^{(d),k}_N)$. Then, the family of vectors
$$
{\cal F} \df \left\{(P_d(i,j)u^{(d),k}_{i+j})_{(i,j)\in{\cal T}_N}:0\leq d\leq N,\ d\leq k\leq N\right\}
$$
is composed of $(N+1)+N+(N-1)+\ldots+1=(N+1)(N+2)/2=|{\cal T}_N|$ elements. Moreover, one can prove that it is
linearly independent as follows: since $\{u^{(d),k}\}_{d\leq k\leq N}$ is a basis of $\CC^{N-d+1}$, it is
sufficient to check that
\begin{equation}
  \label{eq:pf-thm-diago}
  \sum_{d=0}^N P_d(i,j)v^{(d)}_{i+j}=0,\quad\forall (i,j)\in{\cal T}_N
\end{equation}
implies that $v^{(d)}=0$ for all $0\leq d\leq N$, where $v^{(d)}=(v^{(d)}_d,\ldots,v^{(d)}_N)\in\CC^{N-d+1}$,
seen as a subspace of $\CC^{N+1}$ by putting the $d$ first coordinates to be zero. Given $k\leq N$, the
equality~(\ref{eq:pf-thm-diago}) for $i+j=k$ combined with Proposition~\ref{prop:Pd}~(f) yields
$v_k^{(0)}=\ldots=v_k^{(k)}=0$.

Therefore, ${\cal F}$ is a basis of $\CC^{{\cal T}_N}$ and, by point~(a), the matrix $\Pi$ has a Jordan normal
form in this basis. Point~(b) is then staightforward.

If $\wi{\Pi}_0$ admits a positive reversible measure $\mu$, it is straightforward to check
that the vector $\nu$ in~(\ref{eq:def-nu}) is a reversible measure for $\wi{\Pi}$, and hence the first part
of Point~(c) is true.

In addition, the matrix $\Pi_1$ is reversible w.r.t.\ the measure
$$
\mu^{(1)}_n \df 2\,n^2\,\mu_n, \quad n\in{\cal S}_Z^*,
$$
which implies that $\Pi_1$ admits a basis of right eigenvectors orthonormal w.r.t.\ $\mu^{(1)}$. Similarly,
$\wi{\Pi}_0$ admits a basis of right eigenvectors orthonormal w.r.t.\ $\mu$. By Point~(a), this gives $N+1$
(resp.\ $N$) right eigenvectors of $\Pi$ of the form~(\ref{eq:eigenvector-v}) for $d=0$ (resp.\
$d=1$). Together with the basis of right eigenvectors of $\wi{\Pi}$ obtained above (extended by zero on
$\{0\}\times\NN$ and $\NN\times\{0\}$), this gives a basis of eigenvectors of $\Pi$ and ends the proof of~(c).
\end{proof}

\subsubsection{Example: 3-colors urn model (or 3-types Moran model)}
\label{sec:urn}

The class of transition matrices given in~(\ref{eq:neutral-trans}) can be obtained by
composition and linear combinations of the transition matrices
$\Pi^{(n)+}=(\pi^{(n)+}_{(i,j),(k,l)}))_{(i,j),(k,l)\in{\cal S}}$,
$\Pi^{(n)-}=(\pi^{(n)-}_{(i,j),(k,l)}))_{(i,j),(k,l)\in{\cal S}}$ and
$\Pi^{(n)}=(\pi^{(n)}_{(i,j),(k,l)}))_{(i,j),(k,l)\in{\cal S}}$, where for all $n\geq 1$
\begin{gather*}
  \pi^{(n)+}=
  \begin{cases}
    \frac{i}{i+j}, & \text{if\ }k=i+1,\ l=j,\ i+j=n \\
    \frac{j}{i+j}, & \text{if\ }k=i,\ l=j+1,\ i+j=n \\
    0 & \text{otherwise,}
  \end{cases} \\
  \pi^{(n)-}=
  \begin{cases}
    \frac{i}{i+j}, & \text{if\ }k=i-1,\ l=j,\ i+j=n \\
    \frac{j}{i+j}, & \text{if\ }k=i,\ l=j-1,\ i+j=n \\
    0 & \text{otherwise,}
  \end{cases}
\end{gather*}
and for all $n\geq 0$
$$
\pi^{(n)}=
\begin{cases}
  1, & \text{if\ }k=i,\ l=j,\ i+j=n \\
  0 & \text{otherwise.}
\end{cases}
$$

One easily checks, first that the vector spaces ${\cal V}_d$ for all $d\geq 0$ are stable for
all these matrices, and second that, for the matrix~(\ref{eq:neutral-trans}),
$$
\Pi=\sum_{n<m\in{\cal S}_Z}p_{n,m}(\Pi^{(n)+})^{m-n}
+\sum_{n>m\in{\cal S}_Z}p_{n,m}(\Pi^{(n)-})^{n-m}
\sum_{n\in{\cal S}_Z}p_{n,n}\Pi^{(n)}.
$$
Hence, the vector spaces ${\cal V}_d$ are trivially stable for such matrices.

One may however recover a much larger class of matrices for which ${\cal V}_d$ are stable
vector spaces, by considering the algebra of matrices spanned by the matrices $\Pi^{(n)\pm}$
and $\Pi^{(n)}$. Below, we study in detail such an example. \bigskip

Consider an urn with $N$ balls of three different colors and consider the following process:
one picks a ball at random in the urn, notes its color, puts it back in the urn, picks another
ball in the urn and replaces it by a ball of the same color as the first one. The number of
balls of each colors then forms a Markov chain, which can be viewed as the embedded Markov
chain of the 3-types Moran model, defined as follows: consider a population of $N$
individuals, with 3 different types. For each pair of individuals, at rate 1, the second
individual is replaced by an individual of the same type as the first individual in the pair.

Let $i$ denote the number of balls of the first color, and $j$ the number of balls of the
second color. Then, there are $N-i-j$ balls of the third color. The transition probabilities
of this Markov chain are as follows: given that the current state of the process is $(i,j)$,
the state at the next time step is
\begin{equation*}
  \begin{array}{rll}
    (i+1,j) & \text{with probability} & \frac{i(N-i-j)}{N^2}, \\
    (i-1,j) & \text{with probability} & \frac{i(N-i-j)}{N^2}, \\
    (i,j+1) & \text{with probability} & \frac{j(N-i-j)}{N^2}, \\
    (i,j-1) & \text{with probability} & \frac{j(N-i-j)}{N^2}, \\
    (i+1,j-1) & \text{with probability} & \frac{ij}{N^2}, \\
    (i-1,j+1) & \text{with probability} & \frac{ij}{N^2}, \\
    (i,j) & \text{with probability} & \frac{i^2+j^2+(N-i-j)^2}{N^2}.
  \end{array}
\end{equation*}
These transition probabilities do not not have the form~(\ref{eq:neutral-trans}). However, a
variant of Proposition~\ref{prop:stabilite}, and hence of Theorem~\ref{thm:diago-finite},
apply to this process, because of the following observation: let us construct the matrices
$\Pi^+$, $\Pi^-$, $\widehat{\Pi}^+$, $\widetilde{\Pi}^+$ from the matrices
$\Pi^+_0=(p^+_{n,m})_{n,m\in{\cal S}_Z}$, $\Pi^-_0=(p^-_{n,m})_{n,m\in{\cal S}_Z}$,
$\widehat{\Pi}^+_0=(\widehat{p}^+_{n,m})_{n,m\in{\cal S}_Z}$,
$\widetilde{\Pi}^+_0=(\widetilde{p}^+_{n,m})_{n,m\in{\cal S}_Z}$ respectively, exactly as $\Pi$ was
constructed from $\Pi_0$ in~(\ref{eq:neutral-trans}), where
\begin{equation*}
  \begin{array}{lll}
    p^+_{n,n+1}=\frac{k}{N}, & p^+_{n,n}=1-\frac{k}{N}, & p^+_{n,m}=0 \text{\ otherwise}, \\    
    p^-_{n,n-1}=\frac{k}{N}, & p^+_{n,n}=1-\frac{k}{N}, & p^+_{n,m}=0 \text{\ otherwise}, \\    
    \widehat{p}^+_{n,n+1}=\frac{k(N-k)}{N^2}, & \widehat{p}^+_{n,n}=1-\frac{k(N-k)}{N^2}, &
    \widehat{p}^+_{n,m}=0 \text{\ otherwise}, \\    
    \widetilde{p}^+_{n,n+1}=\frac{k(N-k-1)}{N^2}, & \widetilde{p}^+_{n,n}=1-\frac{k(N-k-1)}{N^2}, &
    \widetilde{p}^+_{n,m}=0 \text{\ otherwise}. \\    
  \end{array}
\end{equation*}
Then the transition matrix of the 3-colors urn model is given by
$$
\Pi=\Pi^+\Pi^-+\widehat{\Pi}^+-\widetilde{\Pi}^-.
$$
In particular, the vector spaces ${\cal V}_d$ for $0\leq d\leq N$ are all stable for this
matrix.

The transition matrix $\Pi$ has absorbing sets $\{(i,0):0\leq i\leq N\}$, $\{(0,i):0\leq i\leq
N\}$ and $\{(i,N-i):0\leq i\leq N\}$, and absorbing states $(0,0)$, $(N,0)$ and $(0,N)$. The
restriction of the matrix $\Pi$ on the set 
$$
{\cal S}^{**}:=\{(i,j):i\geq 1,\ j\geq 1,\ i+j\leq N-1\}
$$
admits the reversible measure
$$
\nu_{(i,j)}=\frac{1}{ij(N-i-j)}.
$$
Hence the matrix $\Pi$ admits a family of right eigenvectors null on the absorbing sets, which
forms an orthonormal basis of $\LL^2({\cal S}^{**},\nu)$.

One easily checks, using~(\ref{eq:utile-Moran-3}), that $v_{(i,j)}=P_d(i,j)u_{i+j}$ is a right
eigenvector of $\Pi$ for the eigenvalue $\theta$ if and only if for all $d\leq k\leq N$
\begin{align}
  \theta' u_k & =(N-k)\left[(k+d)u_{k+1}-2k u_k+(k-d)u_{k-1}\right] \notag \\
  & =(N-k)\left[k(u_{k+1}-2k u_k+u_{k-1})+d(u_{k+1}-u_{k-1})\right], \label{eq:Moran-3}
\end{align}
where
$$
\theta=1+\frac{\theta'-d(d-1)}{N^2}.
$$
Now, the Hahn polynomials $Q_n(x;\alpha,\beta,N)$ introduced in~(\ref{eq:Hahn-poly}) satisfy
(cf.\ e.g.~\cite{karlin-mcgregor-75})
\begin{multline*}
  -n(n+\alpha+\beta+1)Q_d(x)=x(N+\beta-x)Q_n(x-1)
  +(N-1-x)(\alpha+1+x)Q_n(x+1)\\
  -[x(N+\beta-x)+(N-1-x)(\alpha+1+x)]Q_n(x).
\end{multline*}
Hence, for all $0\leq n\leq N-d$,~(\ref{eq:Moran-3}) admits the polynomial (in $k$) solution
of degree $n$
$$
u_k=Q_n(k-d;2d-1,-1,N-d+1).
$$
If $n\geq 1$, this polynomial must be divisible by $(N-k)$, so we can define the polynomial
$R_n^{(N,d)}(X)$ of degree $n-1$ as
$$
(N-X)R^{(N,k)}_n(X)=Q_n(X-d;2d-1,-1,N-d+1).
$$
Obviously, the family of vectors $(1,\ldots,1)$ and
$((N-d)R_n^{(N,d)}(d),\ldots,R_n^{(N,d)}(N-1),0)$ for $1\leq n\leq N-d$ is linearly
independent and hence forms a basis of the vector space $\RR^{N-d+1}$ of real vectors indexed
by $d,d+1,\ldots,N$. In addition,~(\ref{eq:Moran-3}) cannot admit any other linearly
independent solution and hence, necessarily, $R_n^{(N,d)}(k)=0$ for all $n>N-d$ and $d\leq
k\leq N-1$.

We have obtained a basis of right eigenvectors of $\Pi$ of the form
$$
\left\{P_d(i,j)\right\}_{0\leq d\leq
  N}\bigcup\left\{P_d(i,j)(N-i-j)R_n^{(N,d)}(i+j)\right\}_{0\leq d\leq N-1,\ 1\leq n\leq N-d},
$$
and the eigenvalue corresponding to the eigenvector $P_d(i,j)$ if $n=0$, or
$P_d(i,j)(N-i-j)R_n^{(N,d)}(i+j)$ if $n\geq 1$, is
$$
\theta_{d,n}:=1-\frac{n(n-1)+2nd-d(d-1)}{N^2}=1-\frac{(d+n)(d+n-1)}{N^2}.
$$
Similarly as in the proof of Proposition~\ref{prop:explicit-Pd}, this family of eigenvectors
can be seen as a singular limit case of those obtained in~\cite{karlin-mcgregor-75} for the
multitype Moran model with mutation or immigration.

Note that in the case of the 2-colors urn model, one can easily check that a basis of right
eigenvectors of the corresponding transition matrix is given by $(1,\ldots,1)$ and
$(NR_n^{(N,0)}(0),\ldots,R_n^{(N,0)}(N-1),0)$ for $1\leq n\leq N$. Hence the spectrum is the
same in the 2- and 3-colors urn models, although the multiplicity of each eigenvalue is
different. In the case of two colors, the eigenvalues have the form $1-k(k-1)/N^2$ for $0\leq
k\leq N$, each with multiplicity 1 (except for the eigenvalue 1, with multiplicity 2). In the
case of three colors, the eigenvalue $1-k(k-1)/N^2$ has multiplicity $k+1$ (except for the
eigenvalue 1, which has multiplicity 3).

Concerning the eigenvectors in $\LL^2({\cal S}^{**},\nu)$. they are given by
$$
\left\{ij(N-i-j)Q_d(i,j)R_n^{(N,d)}(i+j)\right\}_{2\leq d\leq N-1,\ 1\leq n\leq N-d},
$$
and for all $3\leq k\leq N$, the eigenspace for the eigenvalue $1-k(k-1)/N^2$ is
\begin{multline*}
  V_k:=\text{Vect}\{ij(N-i-j)Q_2(i,j)R_{k-2}^{(N,2)}(i+j),\ldots, \\
  ij(N-i-j)Q_{k-1}(i,j)R_1^{(N,k-1)}(i+j)\}.
\end{multline*}

We shall end the study of this example by giving an apparently non-trivial relation between
the polynomials $P_d$ and $R_n^{(N,d)}$. Because of the symmetry of the colors, we have
\begin{align*}
  V_k & =\text{Vect}\{ij(N-i-j)Q_2(i,N-i-j)R_{k-2}^{(N,2)}(N-j),\ldots, \\ & \qquad\qquad\qquad\qquad\qquad
  ij(N-i-j)Q_{k-1}(i,N-i-j)R_1^{(N,k-1)}(N-j)\} \\
  & =\text{Vect}\{ij(N-i-j)Q_2(N-i-j,j)R_{k-2}^{(N,2)}(N-i),\ldots, \\ & \qquad\qquad\qquad\qquad\qquad
  ij(N-i-j)Q_{k-1}(N-i-j,j)R_1^{(N,k-1)}(N-i)\},
\end{align*}
and hence
\begin{align*}
  &
  \text{Vect}\{Q_2(i,j)R_{n-2}^{(N,2)}(i+j),\ldots,Q_{n-1}(i,j)R_1^{(N,n-1)}(i+j)\}_{(i,j)\in{\cal
      S}^{**}} \\ & =
  \text{Vect}\{Q_2(i,N-i-j)R_{n-2}^{(N,2)}(N-j),\ldots,Q_{n-1}(i,N-i-j)R_1^{(N,n-1)}(N-j)\}_{(i,j)\in{\cal
      S}^{**}} \\ & =
  \text{Vect}\{Q_2(N-i-j,j)R_{n-2}^{(N,2)}(N-i),\ldots,Q_{n-1}(N-i-j,j)R_1^{(N,n-1)}(N-i)\}_{(i,j)\in{\cal
      S}^{**}}.
\end{align*}

\subsection{The case of infinite state space}
\label{sec:infinite}

The goal of this subsection is to extend Theorem~\ref{thm:diago-finite} to the case where $Z$ has infinite
state space and the matrix $\Pi_0$ is a compact operator and admits a reversible measure. To this aim, we need first some
approximation properties of $\Pi$ by finite rank operators.

\subsubsection{Approximation properties}
\label{sec:approx}

Recall that $\Pi_0$ is a Markov kernel on $\ZZ_+$ absorbed at 0 and that $\wi \Pi_0$ denotes its restriction
to $\NN$ (equivalently, $\wi \Pi_0$ is the sub-Markovian kernel on $\NN$ corresponding to $\Pi_0$ with a
Dirichlet condition at 0).  We assume that $\wi \Pi_0$ is reversible with respect to a positive measure $\mu$
on $\NN$ (not necessarily finite).  For any $N\in\NN$, consider the sub-Markovian kernel $\wi\Pi_0^{(N)}$ on
$\NN$ defined by
$$
\forall x,y\in\NN,\qquad \wi \Pi_0^{(N)}(x,y)\df
\begin{cases}
\Pi_0(x,y)&\hbox{if $x,y\leq N$}\\
0&\hbox{otherwise.}\end{cases}
$$
In other words, $\wi{\Pi}^{(N)}_0=\text{Pr}_N\wi\Pi_0\text{Pr}_N$, where $\text{Pr}_N$ is the projection operator
defined by $\text{Pr}_N(u_1,u_2,\ldots)=(u_1,\ldots,u_N,0,\ldots)$.
 
The kernel $\wi\Pi_0^{(N)}$ is not Markovian (since $\wi\Pi_0^{(N)}(x,\NN)=0$ for $x>N$),
but it can be seen as the restriction to $\NN$ of a unique Markovian kernel $\Pi_0^{(N)}$
on $\ZZ_+$ absorbed at 0. With this interpretation, we can construct the matrix $\wi\Pi^{(N)}$ from
$\Pi_0^{(N)}$ exactly as the matrix $\wi\Pi$ was constructed from $\Pi_0$ in the beginning of
Section~\ref{sec:diago}.

Of course $\wi\Pi_0^{(N)}$ remains reversible with respect to $\mu$, and thus, like $\wi\Pi_0$, it can be
extended into a self-adjoint operator on $\LL^2(\NN,\mu)$.
We denote by $\vvv\cdot\vvv_0$ the natural operator norm on the set of bounded operators on $\LL^2(\NN,\mu)$,
namely if $K$ is such an operator,
$$
\vvv K\vvv_0 \df \sup_{u\in\LL^2(\NN,\mu)\setminus\{0\}}\frac{\lVe K u\rVe_{\mu}}{\lVe u\rVe_{\mu}}.
$$
If furthermore $K$ is self-adjoint, we have, via spectral calculus,
$$
\vvv K\vvv_0 \df \sup_{u\in\LL^2(\NN,\mu)\setminus\{0\}}\frac{\lve\langle u,Ku\rangle_\mu\rve}{\|u\|^2_\mu}.
$$
The next result gives a simple compactness criterion for $\wi\Pi_0$.

\begin{lem}\label{PiPi}
The operator $\wi\Pi_0$ acting on $\LL^2(\NN,\mu)$
is compact if and only if
$$
\lim_{N\ri \iy}\vvv \wi\Pi_0-\wi\Pi_0^{(N)}\vvv_0 = 0.
$$
\end{lem}

\begin{proof}
Since for any $N\in\ZZ_+$, $\wi\Pi_0^{(N)}$ has finite range, if the above convergence holds, $\wi\Pi_0$ can
be strongly approximated by finite range operators and it is well-known that this implies that $\wi\Pi_0$ is
compact.

The converse implication can be proved adapting a standard argument for compact operators: assume that
$\wi\Pi_0$ is compact and let $\varepsilon >0$ be fixed. Then $\wi{\Pi}_0(B)$ is compact, where $B$ is the closed
unit ball of $\LL^2(\NN,\mu)$. Hence
$$
\wi{\Pi}_0(B)=\bigcup_{i=1}^n B(\psi^{(i)},\varepsilon),
$$
for some $n<+\infty$ and $\psi^{(1)},\ldots,\psi^{(n)}\in\wi{\Pi}_0(B)$, where $B(\psi,\varepsilon)$ is the
closed ball centered at $\psi$ with radius $\varepsilon$. For any $i\leq n$, since
$\psi^{(i)}\in\LL^2(\NN,\mu)$, there exists $N_i$ such that $\sum_{k\geq
  N_i}(\psi^{(i)}_k)^2\mu_k\leq\varepsilon$, and thus
$$
\wi{\Pi}_0(B)=\bigcup_{i=1}^n B(\text{Pr}_{N_i}\psi^{(i)},2\varepsilon).
$$
In other words, for all $\varphi\in B$, there exists $i\leq n$ such that
$\|\wi{\Pi}_0\varphi-\text{Pr}_{N_i}\psi^{(i)}\|_{\mu}\leq 2\varepsilon$. This implies that
$$
\|\wi{\Pi}_0\varphi-\text{Pr}_N\wi{\Pi}_0\varphi\|_\mu\leq 2\varepsilon,
$$
where $N=\sup\{N_1,\ldots,N_n\}$, i.e.\ $\vvv \wi\Pi_0-\text{Pr}_N\wi\Pi_0^{(N)}\vvv_0\leq
2\varepsilon$. Since $\wi\Pi_0^{(N)}=\text{Pr}_N\wi\Pi_0\text{Pr}$, we obtain that 
$$
\lim_{N\rightarrow+\infty}\vvv
\wi\Pi_0\text{Pr}_N-\wi\Pi_0^{(N)}\vvv_0= 0.
$$

In order to complete the proof, it only remains to check that 
$$
\lim_{N\rightarrow+\infty}\vvv \wi\Pi_0-\wi\Pi_0\text{Pr}_N\vvv_0
=0.
$$
If this was false, one could find a sequence $(\varphi^{(N)})_{N\geq 1}$ in $B$ such that $\varphi^{(N)}_k=0$
for all $k\leq N$ and $\|\wi\Pi_0\varphi^{(N)}\|_\mu$ would not converge to 0. Such a sequence
$(\varphi^{(N)})_{N\geq 1}$ weakly converges to 0. Now, another usual characterization of compact operators is
the fact that the image of weakly converging subsequences strongly converges to the image of the limit. In
other words, $\|\wi\Pi_0\varphi^{(N)}\|_\mu\rightarrow 0$. This contradiction ends the proof of
Lemma~\ref{PiPi}.
\end{proof}

The interest of $\wi \Pi_0^{(N)}$ is that it brings us back to the finite situation.
Let $\wit \Pi_0^{(N)}$ be the restriction of $\wi \Pi_0^{(N)}$ to $\lin 1, N\rin$,
which can be seen as a $N\times N$ matrix.
We have for instance that the spectrum of $\wi \Pi_0^{(N)}$ is the spectrum of $\wit \Pi_0^{(N)}$
plus the eigenvalue 0.

We are now going to see how the results of Section~\ref{sec:finite} are affected by the change from $\Pi_0$ to
$\Pi_0^{(N)}$.  More generally, we consider two Markov kernels $\Pi_0$ and $\Pi_0'$ on $\ZZ_+$ absorbed at
0, whose restrictions to $\NN$, $\wi \Pi_0$ and $\wi \Pi_0'$, are both reversible with respect to
$\mu$.  We associate to them $\Pi$ and $\Pi'$ defined on $\ZZ_+^2$ as in (\ref{eq:neutral-trans}), and their
respective restriction to $\NN^2$, $\wi\Pi$ and $\wi \Pi'$. We also define the matrices $\Pi_d$ and $\Pi_d'$
for $d\geq 1$, as in (\ref{eq:def-M_d}). Note that $\wi\Pi$ and $\wi\Pi'$ are reversible with respect to
$\nu$, defined in (\ref{eq:def-nu}) and it is straightforward to check that, for any $d\geq 1$, $\Pi_d$ and
$\Pi_d'$ are both reversible w.r.t.\ $\mu^{(d)}=(\mu^{(d)}_n)_{n\in{\cal S}_Z,\:n\geq d}$ defined by
\begin{equation}
  \label{eq:def-mu-d}
  \mu_n^{(d)} \df 2\,n\,\binom{n+d-1}{2d-1}\,\mu_n,\quad n\in\NN,\quad n\geq d.  
\end{equation}
We will denote $\vvv\cdot\vvv$ and $\vvv\cdot\vvv_d$ the operator norms in $\LL^2(\NN^2,\nu)$ and
$\LL^2(\NN_d,\mu^{(d)})$, where $\NN_d\df\{d,d+1,\ldots\}$. The next result shows that, if one takes
$\Pi_0'=\Pi_0^{(N)}$, the approximation of $\wi\Pi_0$ by $\wi\Pi_0^{(N)}$ behaves nicely.
\begin{prop}
\label{compnorm}
We always have
$$
\vvv \wi\Pi-\wi\Pi'\vvv=
\sup_{d\geq 2}\vvv \Pi_d-\Pi_d'\vvv_d
$$
Furthermore, if $\wi\Pi_0-\wi\Pi'_0\geq0$ (in the sense that all the entries of this infinite matrix are
non-negative), then
\begin{equation}
  \label{vvvd0-1}
  \forall d\geq 1,\quad \vvv \Pi_{d+1}-\Pi'_{d+1}\vvv_{d+1}\leq \vvv \Pi_{d}-\Pi'_{d}\vvv_{d}
\end{equation}
and
\begin{equation}
  \label{vvvd0-2}
  \vvv\Pi_1-\Pi'_1\vvv_1=\vvv\wi\Pi_0-\wi\Pi'_0\vvv_0.
\end{equation}
In particular,
$$
\vvv \wi\Pi-\wi\Pi'\vvv=\vvv\Pi_2-\Pi'_2\vvv_2\leq\vvv \wi\Pi_0-\wi\Pi'_0\vvv_0.
$$
\end{prop}

\begin{proof}
For $d\geq 2$, denote by ${\cal V}'_d$ the set of $v\in\LL^2(\NN^2,\nu)$ of the form
$$
\forall i,j\in\NN,\qquad v_{i,j}=P_d(i,j)u_{i+j}
$$
with $u\in \LL^2(\NN_d,\mu^{(d)})$. We denote by $v(d,u)$ the sequence $v$ defined by the above r.h.s.  The
definitions of $\nu$, $\mu^{(d)}$ and Proposition~\ref{prop:Pd}~(h) enable us to see that the mapping
$$
\LL^2(\NN_d,\mu^{(d)})\ni u \ \mapsto \ v(d,u)\in \LL^2(\NN^2,\nu)
$$
is an isometry.

Proposition~\ref{prop:Pd}~(h) also shows that ${\cal V}'_d$ and ${\cal V}'_{d'}$ are orthogonal subspaces of
$\LL^2(\NN^2,\nu)$ for all $d,d'\geq 2$, $d\not=d'$. We actually have
\begin{equation}
  \label{eq:somme-directe}
  \LL^2(\NN^2,\nu)=\bigoplus_{d\geq2}{\cal V}'_d.  
\end{equation}
Indeed, let $v\in\LL^2(\NN^2,\nu)$ be orthogonal to ${\cal V}'_d$ for all $d\geq 2$. For all $d\geq 2$, we
define the vector
$$
v^{(d)}_l=\frac{1}{\binom{l+d-1}{2d-1}}\sum_{i=1}^{l-1}v_{i,l-i}\frac{P_d(i,l-i)}{i(l-i)}, \quad l\geq d.
$$
The Cauchy-Schwartz inequality and Proposition~\ref{prop:Pd}~(h) imply that $v^{(d)}\in\LL^2(\NN_d,\mu^{(d)})$ and
since $v$ is orthogonal to ${\cal V}'_d$, by the definition of ${\cal V}'_d$, the vector $v^{(d)}$ is orthogonal
to $\LL^2(\NN_d,\mu^{(d)})$, i.e.\ $v^{(d)}=0$. Fixing $l\geq 2$ and applying Proposition~\ref{prop:Pd}~(f),
one deduces from the equations $v^{(d)}_l=0$ for $2\leq d\leq l$ that $v_{i,l-i}=0$ for $1\leq i\leq l-1$, and
thus $v=0$, ending the proof of~(\ref{eq:somme-directe}).

Now, Proposition~\ref{prop:stabilite} show that ${\cal V}'_d$ is stable by $\wi \Pi$ and $\wi
\Pi'$ and, more precisely,
$$
\forall u\in \LL^2(\NN_d,\mu^{(d)}),\quad \wi\Pi[v(d,u)]=v(d,\Pi_d u)\quad\hbox{and}\quad
\wi\Pi'[v(d,u)]=v(d,\Pi'_d u).
$$
It then follows from~(\ref{eq:somme-directe}) that
$$
\vvv \wi\Pi-\wi\Pi'\vvv= \sup_{d\geq 2}\vvv \Pi_d-\Pi_d'\vvv_d.
$$

Let us now come to the proof of~(\ref{vvvd0-1}). Let $d\geq1$ be fixed and define
$$
g^{(d)} = \frac{d\mu^{(d+1)}}{d\mu^{(d)}}\qquad \hbox{and}\qquad h^{(d)} =
\frac{d(\Pi_{d+1}-\Pi_{d+1}')}{d(\Pi_d-\Pi'_d)},
$$
i.e.
$$
\forall n\geq d+1,\qquad g^{(d)}_n \df \frac{(n-d)(n+d)}{2d(2d+1)}
$$
and
$$
\forall i,j\geq d+1,\qquad h^{(d)}_{i,j} \df
\begin{cases}
  \displaystyle{\frac{j+d}{i+d} }& \text{if\ }j>i, \\
  \displaystyle{\frac{j-d}{i-d}} & \text{if\ }j<i, \\
  1 & \text{if\ }i=j.
\end{cases}
$$
For any $u\in \LL^2(\NN_{d+1},\mu^{(d+1)})$, we get
\begin{align*}
  \lve\frac{\langle u,(\Pi_{d+1}-\Pi_{d+1}')u\rangle_{\mu^{(d+1)}}}{\|u\|^2_{\mu^{(d+1)}}}\rve & \leq
  \frac{\sum_{i,j\in\NN}\mu_i^{(d)}\,(\Pi_d-\Pi'_d)_{i,j}\,g^{(d)}_i\,h^{(d)}_{i,j}\,|u_i|\,|u_j|}
  {\sum_{i\in\NN}\mu^{(d)}_i\,g^{(d)}_i\,u^2_i}\\
  &= \frac{\langle \wi u,(\Pi_d-\Pi_d')\wi u\rangle_{\mu^{(d)}}}{\|u\|^2_{\mu^{(d)}}}\,
  \sup_{i,j\in\NN_{d+1}}\sqrt{\frac{g^{(d)}_i}{g^{(d)}_j}}\,h^{(d)}_{i,j},
\end{align*}
where
$$
\forall i\geq d,\qquad \wi u_i \df \sqrt{g^{(d)}_i}\,|u_i(x)|,
$$
and $\wi u_i=0$ if $i=d$. It is clear that $\wi u\in\LL^2(\NN_d,\mu^{(d)})$, so, taking the supremum over
$u\in \LL^2(\NN_{d+1},\mu^{(d+1)})$, we have
$$
\vvv \Pi_{d+1}-\Pi_{d+1}'\vvv_{d+1}
\leq\sup_{i,j\in\NN_{d+1}}\sqrt{\frac{g^{(d)}_i}{g^{(d)}_j}}\,h^{(d)}_{i,j}\,\vvv \Pi_d-\Pi'_d\vvv_d.
$$
Hence, it only remains to show that 
$$
\sup_{i,j\in\NN_{d+1}}\sqrt{\frac{g^{(d)}_i}{g^{(d)}_j}}\,h^{(d)}_{i,j}\leq 1.
$$
Since $\sqrt{\frac{g^{(d)}_i}{g^{(d)}_i}}h^{(d)}_{i,i}=1$ for all $i\geq 1$, it is sufficient to prove that
$$
\forall j>i\geq d+1,\qquad \sqrt{\frac{g^{(d)}_i}{g^{(d)}_j}}\,h^{(d)}_{i,j} \leq 1
$$
since the l.h.s.\ is symmetrical in $i,j$.  For $j>i\geq d+1$, we compute
$$
\sqrt{\frac{g^{(d)}_i}{g^{(d)}_j}}\,h^{(d)}_{i,j} = \sqrt{\frac{(i-d)(j+d)}{(j-d)(i+d)}}\leq 1,
$$
ending the proof of~(\ref{vvvd0-1}). A similar computation using the fact that
$$
g^{(0)} \df \frac{d\mu^{(1)}}{d\mu}\qquad \hbox{and}\qquad h^{(0)} \df
\frac{d(\Pi_{1}-\Pi_{1}')}{d(\wi\Pi_0-\wi\Pi'_0)}
$$
are given by
$$
\forall i,j\geq 1,\qquad g^{(0)}_i=2i^2\qquad\text{and}\qquad h^{(0)}_{i,j}=\frac{j}{i}
$$
leads to~(\ref{vvvd0-2}).
\end{proof}

Again, let $\wit \Pi^{(N)}$ be the restriction of $\wi \Pi^{(N)}$ to $\cT_N^*$, which can thus be seen as a
finite $\cT_N^*\times \cT_N^*$ matrix.  Similarly to the remark after the proof of Lemma~\ref{PiPi}, the
spectrum of $\wi \Pi^{(N)}$ is the spectrum of $\wit \Pi^{(N)}$ plus the eigenvalue 0.

\subsubsection{Spectral decomposition of infinite, compact, reversible N2dMC}
\label{sec:applic}

The following result is an immediate consequence of Lemma~\ref{PiPi} and Proposition~\ref{compnorm}.
\begin{cor}
  If $\wi \Pi_0$ is compact and reversible, the same is true for $\wi \Pi$.
\end{cor}

We can now extend Theorem~\ref{thm:diago-finite} to the infinite compact, reversible case.
\begin{thm}
  \label{thm:diago-infinite} Assume that $Z$ has an infinite state space, i.e.\ ${\cal S}_Z=\ZZ_+$.
  \begin{description}
  \item[\textmd{(a)}] Theorem~\ref{thm:diago-finite}~(a) also holds true in this case.
  \item[\textmd{(b)}] If $\wi{\Pi}_0$ is compact and reversible w.r.t.\ the measure $(\mu_n)_{n\in\NN}$, then,
    there exists an orthonormal basis of $\LL^2(\NN^2,\nu)$ of right eigenvectors of $\wi{\Pi}$ of the
    form~(\ref{eq:eigenvector-v}) for $d\geq 2$, where $\nu$ is defined in~(\ref{eq:def-nu}).
    Moreover, any right eigenvector of $\wi\Pi$ in $\LL^2(\NN^2,\nu)$ is a linear combination of
    eigenvectors of the form~(\ref{eq:eigenvector-v}) all corresponding to the same eigenvalue.
  \item[\textmd{(c)}] Under the same assumptions as~(b), there exists a family of right eigenvectors of $\Pi$
    of the form
    \begin{equation}
      \label{eq:basis-discrete}
      \{1\}\cup\Big\{P^{(1)}_1(i,j)u^{(1),l}_{i+j},\:P^{(2)}_1(i,j)u^{(1),l}_{i+j}\Big\}_{l\geq 1}\cup
      \bigcup_{d\geq 2}\Big\{P_d(i,j)u^{(d),l}_{i+j}\Big\}_{l\geq 1},
    \end{equation}
    which is a basis of the vector space
    \begin{multline}
      V \df \Big\{v\in\RR^{\ZZ_+^2}:
      v_{i,j}=a+\frac{i}{i+j}v^{(1)}_{i+j}+\frac{j}{i+j}v^{(2)}_{i+j}+v^{(3)}_{i,j},\ \forall i,j\in\ZZ_+, \\
      \text{\ with\ }a\in\RR,\ v^{(1)},v^{(2)}\in\LL^2(\NN,\mu)\text{\ and\
      }v^{(3)}\in\LL^2(\NN^2,\nu)\Big\} \label{eq:def-V-discrete}
    \end{multline}
    in the sense that, for all $v\in V$, there exist unique sequences $\{\alpha_l\}_{l\geq 1}$,
    $\{\beta_l\}_{l\geq 1}$, and $\{\gamma_{dl}\}_{d\geq 2,\ l\geq 1}$ and a unique $a\in\RR$ such that
    \begin{equation}
      \label{eq:basis-V-discrete}
      v_{i,j}=a+\sum_{l\geq 1}\alpha_l iu^{(1),l}_{i+j}+\sum_{l\geq 1}\beta_l ju^{(1),l}_{i+j}+\sum_{d\geq 2\
        l\geq 1}\gamma_{dl}P_d(i,j)u^{(d),l}_{i+j},
    \end{equation}
    where the series $\sum_l\alpha_lku^{(1),l}_k$ and $\sum_l\beta_lku^{(1),l}_k$ both converge for
    $\|\cdot\|_{\mu}$ and $\sum_{d,l}\gamma_{dl}P_d(i,j)u^{(d),l}_{i+j}$ converges for $\|\cdot\|_\nu$.
  \end{description}
\end{thm}


\begin{example}
  \label{ex:BD}
  Assume that $Z_n=X_n+Y_n$ is a birth and death process, i.e.\ that the matrix $\Pi_0$ is tridiagonal. Assume
  moreover that all the entries just above or below the diagonal are positive (except of course for $p_{0,1}$,
  which is $0$ since $p_{0,0}=1$). It is well-known in this case that there always exists a reversible measure
  $\mu$ for $\wi{\Pi}_0$. A well-known sufficient condition for the compactness of $\wi{\Pi}_0$ is the case
  where this operator is Hilbert-Schmidt, which translates in our reversible, discrete case as
  \begin{equation}
    \label{eq:Hilbert-Schmidt}
    \sum_{i,j\in\NN}p_{i,j}p_{j,i}<\infty.    
  \end{equation}
  For a birth and death process, letting $p_k$ (resp.\ $q_k$) denote the birth (resp.\ death) probability in
  state $k$ and $r_k \df 1-p_k-q_k\geq 0$, this gives
  $$
  \sum_{i\geq 0}\Big(r_i^2+p_iq_{i+1}\Big)<\infty.
  $$
  As a side remark, note that $\wi\Pi$ is not necessarily Hilbert-Schmidt when $\wi\Pi_0$ is, as the condition
  $\sum_{(i,j),(k,l)}\pi_{(i,j),(k,l)}\pi_{(k,l),(i,j)}<\infty$ is not equivalent
  to~(\ref{eq:Hilbert-Schmidt}).
\end{example}

\begin{proof}
Point~(a) can be proved exactly the same way as Theorem~\ref{thm:diago-finite}~(a).

The fact that compact selfadjoint operators admit an orthonormal basis of eigenvectors is classical. To
prove~(b), we only have to check that all these eigenvectors can be taken of the form~(\ref{eq:eigenvector-v})
for $d\geq 2$. This easily follows from the fact that $\Pi_d$ is compact selfadjoint in
$\LL^2(\NN_d,\mu^{(d)})$ for all $d\geq 2$, from~(\ref{eq:somme-directe}) and from Point~(a).

The proof of~(c) is similar to the proof of Theorem~\ref{thm:eigen-SDE}~(b). Fix $v\in V$ and define $a,\
v^{(1)},\ v^{(2)}$ and $v^{(3)}$ as in~(\ref{eq:def-V-discrete}). Since $\mu^{(1)}_k=k^2\mu_k/2$,
$(v^{(1)}_k/k)\in\LL^2(\NN,\mu^{(1)})$ and there exists $(\alpha_l)_{l\geq 1}$ such that
\begin{equation}
  \label{eq:cv-premier-serie}
  v^{(1)}_k=k\sum_{l\geq 1}\alpha_l u^{(1),l}_k=\sum_{l\geq 1}\alpha_l ku^{(1),l}_k, \quad\forall k\in{\cal S}_Z^*,  
\end{equation}
where the convergence of the first series holds for $\|\cdot\|_{\mu^{(1)}}$, and thus of the second series for
$\|\cdot\|_\mu$. A similar decomposition for $v^{(2)}$ and the use of~(b) for $v^{(3)}$ complete the proof
of~(\ref{eq:basis-V-discrete}).

It only remains to observe that, for all $v\in V$, the equation
$$
v_{i,j}=a+\frac{i}{i+j}v^{(1)}_{i+j}+\frac{j}{i+j}v^{(2)}_{i+j}+v^{(3)}_{i,j}
$$
uniquely characterizes $a\in\RR$, $v^{(1)},\ v^{(2)}\in\LL^2(\NN,\mu)$ and
$v^{(3)}\in\LL^2(\NN^2,\nu)$. Indeed, since $v^{(3)}_{0,j}=v^{(3)}_{i,0}=0$, one must have $a=v_{0,0}$,
$v^{(1)}_i=v_{i,0}-a$ and $v^{(2)}_j=v_{0,j}-a$.
\end{proof}

\section{On Dirichlet eigenvalues in the complement of triangular subdomains}
\label{sec:Dirichlet}

In this section, we consider the same model as in the previous section, and we assume either that ${\cal
  S}={\cal T}_N$ is finite or that ${\cal S}=\ZZ_+^2$ and the restriction $\wi{\Pi}_0$ of $\Pi_0$ to $\NN$ is
compact and reversible w.r.t. some measure $\mu$. We recall that ${\cal T}_k=\{(i,j)\in\ZZ_+^2:i+j\leq k\}$
and ${\cal T}^*_k={\cal T}_k\cap\NN^2$ and we define ${\cal S}^*_k \df {\cal S}^*\setminus{\cal T}^*_{k-1}$
for all $k\geq 2$. Note that ${\cal S}^*_2={\cal S}^*$. We also define ${\cal S}^*_1 \df {\cal
  S}\setminus(\{0\}\times\ZZ_+)$. Finally, for $k\geq 0$, we call $\wi\Pi_k$ the restriction of the matrix
$\Pi_0$ to $\{i\in{\cal S}_Z:i>k\}$, and for $k\geq 1$, $\wit\Pi_k$ the restriction of the matrix $\Pi$ to
${\cal S}_k^*$. Note that this notation is consistent with the previous definition of $\wi\Pi_0$ and that
$\wit\Pi_2=\wi\Pi$. Again, all the notations of this section are gathered for reference in
Appendix~\ref{sec:nota}.

\subsection{The case of finite state space}
\label{sec:finite-2}

Let us first assume that ${\cal S}={\cal T}_N$ for some $N\geq 1$.

For $1\leq k\leq N$, the Dirichlet eigenvalue problem for the matrix $\Pi$ in the set ${\cal S}^*_k$ consists
in finding $\theta\in\CC$ and $v$ in $\RR^{{\cal T}_N}$, such that
\begin{equation}
  \label{eq:Dirichlet}
  \begin{cases}
    (\Pi v)_{(i,j)}=\theta v_{(i,j)} & \forall (i,j)\in{\cal S}^*_k, \\
    v_{(i,j)}=0 & \forall (i,j)\in{\cal S}\setminus{\cal S}^*_k.
  \end{cases}
\end{equation}
This is equivalent to finding a right eigenvector of $\wit\Pi_k$ and extending this vector by $0$ to indices
in ${\cal S}\setminus{\cal S}^*_k$. For all $k\geq 1$, we define $\theta^D_k$ as the supremum of the moduli of
all Dirichlet eigenvalues in ${\cal S}^*_k$. By Perron-Fr\"obenius' theory, $\theta^D_k$ is a Dirichlet
eigenvalue in ${\cal S}_k^*$.

For all $d\geq 0$, we also define $\theta^{(d)}$ as the supremum of the moduli of all eigenvalues of $\Pi$
corresponding to right eigenvectors of the form $P_d(i,j)u_{i+j}$. Again, by
Theorem~\ref{thm:diago-finite}~(a) and Perron-Fr\"obenius' theory, $\theta^{(d)}$ is an eigenvector of $\Pi$
of the form $P_d(i,j)u_{i+j}$. Note that, for all $d\geq 2$, because of Proposition~\ref{prop:Pd}~(c) and~(e),
any right eigenvector of $\Pi$ of the form $P_d(i,j)u_{i+j}$ is a Dirichlet eigenvector of $\Pi$ in ${\cal
  S}^*_{d'}$ for all $1\leq d'\leq d$. In particular, $\theta^D_{d'}\geq\theta^{(d)}$. The next result gives
other inequalities concerning these eigenvealues.

\begin{thm}
  \label{thm:Dirichlet-finite}
  Assume that ${\cal S}={\cal T}_N$. 
  \begin{description}
  \item[\textmd{(a)}] Then,
    $$
    \begin{array}{rcccccccccl}
      & & \theta^D_1 & \geq & \theta^D_2 & \geq & \theta^D_3 & \geq \ldots \geq & \theta^D_{N-1} & \geq &
      \theta^D_N \\[1mm] & & \rotatebox[origin=c]{90}{$=$} & & \rotatebox[origin=c]{90}{$=$} & &
      \rotatebox[origin=c]{90}{$\leq$} & \ldots & \rotatebox[origin=c]{90}{$\leq$} & &
      \ \rotatebox[origin=c]{90}{$=$} \\[2mm]  
      1=\theta^{(0)} & \geq & \theta^{(1)} & \geq & \theta^{(2)} & \geq & \theta^{(3)} & \geq\ldots\geq &
      \theta^{(N-1)} & \geq & \theta^{(N)}=p_{N,N}.
    \end{array}
    $$
  \item[\textmd{(b)}] If $\wi{\Pi}_{k-1}$ is irreducible for some $1\leq k\leq N-1$, then
    \begin{equation}
      \label{eq:Dirichlet-strict-1}
      \theta^D_{k}>\theta^D_{k+1},\qquad \theta^{(k)}>\theta^{(k+1)}      
    \end{equation}
    and
    \begin{equation}
      \label{eq:Dirichlet-strict-2}
      \theta^D_k>\theta^{(k)},\qquad\text{if\ }k\geq 3.      
    \end{equation}
    If $\wi\Pi_0$ is irreducible and $p_{i,0}>0$ for some $1\leq i\leq N$, then $\theta^{(1)}<1$.
  \end{description}
\end{thm}

\begin{proof}
Since the matrix $\Pi_0$ is stochastic, it is clear from Theorem~\ref{thm:diago-finite}~(a) that
$\theta^{(0)}=1$. By~(\ref{eq:def-M_d}), $p^{(1)}_{n,m}=\frac{m}{n}p_{n,m}$ for all $1\leq n,m\leq N$ and
thus $\wi\Pi_0 u=\theta u$ if and only if $\Pi_1 v=\theta v$, where $v_n=u_n/n$. Hence, the largest
eigenvalue of $\wi\Pi_0$ is $\theta^{(1)}$. Since
$$
\Pi_0=
\begin{pmatrix}
  1 & 0 \\ v & \wi\Pi_0
\end{pmatrix}
$$
for some (column) vector $v\in\RR_+^{N}$, it is clear that $\theta^{(1)}\leq\theta^{(0)}$.  Ordering
conveniently the states in ${\cal S}^*_1$, the matrix $\wit\Pi_1$ has the form
\begin{equation}
  \label{eq:Pi^(1)}
  \wit\Pi_1=\begin{pmatrix}
    \wi{\Pi}_0 & 0 \\ Q & \wit\Pi_2
  \end{pmatrix}
\end{equation}
for some rectangular nonnegative matrix $Q$, since the set $\{0,1,\ldots,N\}\times\{0\}$ is absorbing for the
Markov chain. Again, since this matrix is block triangular, we have $\theta^D_1=\max\{\theta^{(1)},\theta^D_2\}$.
Since $\wit\Pi_2=\wi\Pi$, Theorem~\ref{thm:diago-finite}~(b) shows that 
$$
\theta^D_2=\sup_{k\geq
  2}\theta^{(k)}.
$$
Since in addition $\wit\Pi_N=p_{N,N}\text{Id}$ and $\Pi_N=p_{N,N}$, Theorem~\ref{thm:Dirichlet-finite}~(a)
will hold true if we prove that the sequences $(\theta^{(k)})_{1\leq k\leq N}$ and $(\theta^D_k)_{2\leq k\leq
  N}$ are both non-increasing.

By Perron-Fr\"obenius' characterization of the spectral radius of nonnegative matrices
(cf.\ e.g.~\cite{gantmacher-86}), for all $1\leq k\leq N$,
\begin{equation}
  \label{eq:PF-characterization-1}
  \theta^{(k)}=\sup_{u\in\RR^{\NN_k},\,u\geq 0,\,u\not=0}\ \inf_{i\in\NN_k}\frac{\sum_{j\in\NN_k}p^{(k)}_{i,j}u_j}{u_i},
\end{equation}
where, by convention, the fraction in the r.h.s.\ is $+\infty$ if $u_i=0$. Using the notation of the proof of
Proposition~\ref{compnorm}, for all $1\leq k\leq N-1$ and $u\in\RR_+^{\NN_{k+1}}\setminus\{0\}$, we have
\begin{align*}
  \inf_{i\in\NN_{k+1}}\frac{\sum_{j\in\NN_{k+1}}p^{(k+1)}_{i,j}u_j}{u_i} & =
  \inf_{i\in\NN_{k+1}}\frac{\sum_{j\in\NN_{k+1}}p^{(k)}_{i,j}\wi
    u_j\sqrt{g^{(k)}_i/g^{(k)}_j}h^{(k)}_{i,j}}{\wi u_i} \\ & \leq
  \inf_{i\in\NN_{k}}\frac{\sum_{j\in\NN_{k}}p^{(k)}_{i,j}\wi
    u_j}{\wi u_i}.
\end{align*}
Taking the supremum over $u\in\RR_+^{\NN_{k+1}}\setminus\{0\}$ yields 
\begin{equation}
  \label{eq:ineq-PF-1}
  \theta^{(k+1)}\leq\theta^{(k)}.  
\end{equation}

For all $k\geq 2$, the Dirichlet eigenvectors in ${\cal S}_k^*$ belong to the vector space
$$
{\cal U}_k \df \{v\in\RR^{{\cal T}_N}:v=0\ \text{on\ }{\cal T}_N\setminus{\cal S}^*_k\}.
$$
By Perron-Fr\"obenius' theory again, for all $k\geq 2$,
\begin{equation}
  \label{eq:PF-characterization}
  \theta^D_k=\sup_{w\in {\cal U}_k\setminus\{0\},\ w\geq 0}\ \inf_{(i,j)\in{\cal S}^*_k}\frac{\sum_{(k,l)\in{\cal
        S}^*_k}\pi_{(i,j),(k,l)}w_{(k,l)}}{w_{(i,j)}}.
\end{equation}
Since ${\cal S}^*_k\subset{\cal S}^*_{k-1}$, for all $2\leq k\leq N-1$,
\begin{equation}
  \label{eq:ineq-PF}
  \theta^D_{k+1}=\sup_{w\in {\cal U}_{k+1}\setminus\{0\},\ w\geq 0}\inf_{(i,j)\in{\cal
      S}^*_{k}}\frac{\sum_{(k,l)\in{\cal S}^*_{k}}\pi_{(i,j),(k,l)}w_{(k,l)}}{w_{(i,j)}}\leq\theta^D_{k}.
\end{equation}
This ends the proof of~(a).

In the case where $\wi{\Pi}_{k-1}$ is irreducible for some $2\leq k\leq N-1$, it is clear that $\wit\Pi_k$ and
$\Pi_k$ are both irreducible. Then, by Perron-Fr\"obenius' theory, $\theta^D_k$ (resp.\ $\theta^{(k)}$) is an
eigenvalue of $\wit\Pi_k$ (resp.\ $\Pi_k$) with multiplicity one, and the corresponding nonnegative
eigenvector has all its entries positive. In addition, $\theta^D_k$ (resp.\ $\theta^{(k)}$) is the only
eigenvalue of $\wit\Pi_k$ (resp.\ $\Pi_k$) corresponding to a positive eigenvector. In particular, the
supremum in~(\ref{eq:PF-characterization}) (resp.~(\ref{eq:PF-characterization-1})\,) is attained only at
vectors $w\in V^{(k-1)}$ having all coordinates corresponding to states in ${\cal S}^*_{k-1}$ positive (resp.\
at vectors $u\in(0,\infty)^{\NN_k}$). Hence, the inequalities in~(\ref{eq:ineq-PF}) and~(\ref{eq:ineq-PF-1})
are strict. 

In the case where $\wi{\Pi}_0$ is irreducible, the same argument shows that $\theta^{(1)}>\theta^{(2)}$, and
since $\theta^{(1)}=\theta^D_1$ and $\theta^{(2)}=\theta^D_2$, we also have $\theta^D_1>\theta^D_2$. This ends
the proof of~(\ref{eq:Dirichlet-strict-1}).

In the case where $\wi\Pi_{k-1}$ is irreducible for $3\leq k\leq N-1$, let $(u_i)_{i\in\NN_k}$ be a positive
right eigenvector of $\Pi_k$. Then the vector $P_k(i,j)u_{i+j}$ belongs to ${\cal U}_k$ and its restriction to
${\cal S}^*_k$ is a right eigenvector of $\wit\Pi_k$. However, its has positive and negative coordinates by
Proposition~\ref{prop:Pd}~(g). Therefore~(\ref{eq:Dirichlet-strict-2}) is proved.

Finally, if $\wi{\Pi}_0$ is irreducible and $p_{i,0}>0$ for some $1\leq i\leq N$, then the absorbing state $0$
is accessible by the Markov chain from any initial state (possibly after several steps). It is then standard
to prove that there exists $n$ such that the sums of the entries of each line of $(\wi\Pi_0)^n$ is strictly
less than 1. This proves that $(\wi\Pi_0)^n$ cannot have 1 as eigenvalue, and thus $\theta^{(1)}<1$.
\end{proof}

\subsection{The case of infinite state space}
\label{sec:infinite-2}

Our goal here is to extend the previous result to the case where ${\cal S}=\ZZ_+^2$ and $\wi\Pi_0$ is compact
reversible. So let us assume that $\wi\Pi_0$ is compact and is reversible w.r.t.\ some measure $\mu$.

In the case of infinite state space, when $k\geq 2$ the Dirichlet eigenvalue problem for $\Pi$ in ${\cal
  S}_k^*$ consists in finding $\theta\in\CC$ and $v\in\LL^2(\ZZ_+^2,\nu)$ (where the measure $\nu$ is extended
by convention by zero on $\ZZ_+\times\{0\}\cup\{0\}\times\ZZ_+$) satisfying~(\ref{eq:Dirichlet}). Defining the
vector space where Dirichlet eigenvectors are to be found
$$
{\cal U}_k \df \{v\in\LL^2(\ZZ_+^2,\nu):v=0\ \text{on\ }\ZZ_+^2\setminus{\cal S}^*_k\},
$$
the supremum $\theta^D_k$ of the moduli of all Dirichlet eigenvalues in ${\cal S}_k^*$ is given by
\begin{equation}
  \label{eq:caract-theta-D}
  \theta_k^{D} = \sup_{u\in {\cal U}_k\setminus\{0\}}\frac{\left| \langle u,\wi\Pi u\rangle_\nu\right|}{\|u\|^2_\nu}.  
\end{equation}

In view of Theorem~\ref{thm:diago-infinite}~(c), the natural space to define the Dirichlet eigenvalue problem
in ${\cal S}_1^*$ is
\begin{multline*}
  {\cal U}_1 \df \Big\{v\in\RR^{\ZZ_+^2}:v=0\text{\ in\ }\{0\}\times\ZZ_+\text{\ and\
    }v_{i,j}=\frac{i}{i+j}v^{(1)}_{i+j}+v^{(3)}_{i,j}\ \forall(i,j)\in{\cal S}^*_1, \\
    \text{\ where\ }v^{(1)}\in\LL^2(\NN,\mu)\text{\ and\ }v^{(3)}\in\LL^2(\NN^2,\nu)\Big\},
\end{multline*}
equipped with the norm $\|\cdot\|_{{\cal U}_1}$, where $\|v\|_{{\cal
    U}_1}^2=\|v^{(1)}\|^2_\mu+\|v^{(3)}\|^2_\nu$ (this norm is well-defined since $v^{(1)}$ and $v^{(3)}$ are
uniquely defined from $v\in {\cal U}_1$). Then, the Dirichlet eigenvalue problem in ${\cal S}_1^*$ consists in
finding $\theta\in\CC$ and $v\in{\cal U}_1$ satisfying~(\ref{eq:Dirichlet}). We also define $\theta^D_1$ as
the supremum of the moduli of all Dirichlet eigenvalues in ${\cal S}_1^*$.

For all $d\geq 1$, we also define $\theta^{(d)}$ as the supremum of the moduli of all eigenvalues of $\Pi$
corresponding to right eigenvectors of the form $P_d(i,j)u_{i+j}$ with $u\in\LL^2(\NN_d,\mu^{(d)})$. By
Theorem~\ref{thm:diago-infinite}~(a),
$$
\theta^{(d)} = \sup_{u\in\LL^2(\NN_d,\mu^{(d)})\setminus\{0\}}\frac{\lve \langle u,\Pi_d u\rangle_{\mu^{(d)}}
  \rve}{\|u\|^2_{\mu^{(d)}}},\quad\forall d\geq 0.
$$
In addition, using the notation ${\cal V}'_d$ defined in the proof of Proposition~\ref{compnorm}, it follows
from~(\ref{eq:somme-directe}) that
$$
\theta^{(d)} =\sup_{u\in{\cal V}'_d}\frac{\lve \langle u,\wi\Pi u\rangle_\nu\rve}{\|u\|^2_\nu},\quad\forall d\geq 2.
$$
Comparing this with~(\ref{eq:caract-theta-D}), we again deduce from Proposition~\ref{prop:Pd}~(c) and~(e) that
$\theta^D_{d'}\geq\theta^{(d)}$ for all $2\leq d'\leq d$.

Finally, since the matrix $\Pi_0$ is not reversible, we need to define $\theta^{(0)}$ in a slightly different
way: $\theta^{(0)}$ is the supremum of the moduli of all eigenvalues of $\Pi$ corresponding to right
eigenvectors of the form $P_d(i,j)u_{i+j}$ with $u=a\mathbf{1}+v$, where $a\in\RR$, $\mathbf{1}$ is the vector
of $\RR^{\ZZ_+}$ with all coordinates equal to 1, $v\in\LL^2(\NN,\mu)$ with the convention $v_0=0$.

\begin{thm}
  \label{thm:Dirichlet-infinite}
  Assume that ${\cal S}=\ZZ_+^2$ and that $\wi\Pi_0$ is compact and reversible w.r.t.\ a positive measure
  $\mu$. Then, for all $d\geq 1$, $\theta^D_d$ is a Dirichlet eigenvalue of $\Pi$ in the set ${\cal S}^*_d$
  and $\theta^{(d)}$ is a right eigenvalue of $\Pi$ for an eigenvector of the form
  $P_d(i,j)u_{i+j}$ with $u\in\LL^2(\NN_d,\mu^{(d)})$. In addition,
  $$
  \begin{array}{rcccccccccl}
    & & \theta^D_1 & \geq & \theta^D_2 & \geq & \theta^D_3 & \geq & \theta^D_4 & \geq\ldots \\[1mm] & &
    \rotatebox[origin=c]{90}{$=$} & & \rotatebox[origin=c]{90}{$=$} & &
    \rotatebox[origin=c]{90}{$\leq$} & & \rotatebox[origin=c]{90}{$\leq$} & \\[2mm]  
    1=\theta^{(0)} & \geq & \theta^{(1)} & \geq & \theta^{(2)} & \geq & \theta^{(3)} & \geq & \theta^{(4)} &
    \geq\ldots
  \end{array}
  $$
\end{thm}

\begin{proof}
For all $k\geq 1$, the fact that $\theta^{(k)}$ is an eigenvalue of $\Pi$ for a right eigenvector of the
form $P_k(i,j)u_{i+j}$ with $u\in\LL^2(\NN_k,\mu^{(k)})$ follows from Theorem~\ref{thm:diago-infinite}~(a)
and from the fact that $\Pi_k$ is compact reversible. Indeed, $\vvv\Pi_k\vvv_k$ is an eigenvalue of the
compact reversible matrix $\Pi_k$ (the corresponding eigenvector can be obtained for example as the limit of
nonnegative eigenvectors of finite dimensional projections of $\Pi_k$). The result follows since,
necessarily, $|\theta^{(k)}|\leq\vvv\Pi_k\vvv_k$.

For all $k\geq 2$, we define $\Pi^{(k)}$ as the matrix whose restriction to ${\cal S}_k^*$ is $\wit\Pi_k$ and
with all other coordinates equal to zero. As the projection of a compact operator, this matrix is
compact. Since it is trivially reversible for the measure $\nu$, Theorem~\ref{thm:diago-infinite}~(b) applies
to $\Pi^{(k)}$. Then, Proposition~\ref{compnorm} tells us that $\theta^D_k=\vvv\Pi_2^{(k)}\vvv_2$, where
$\Pi_2^{(k)}$ is defined from $\Pi^{(k)}$ as $\Pi_2$ has been defined from $\Pi$. Therefore, the fact that
$\theta^D_k$ is a Dirichlet eigenvalue for $\Pi$ in ${\cal S}_k^*$ can be deduced exactly as above.

Recall the definition of $\wi\Pi^{(N)}$ in Section~\ref{sec:approx}. For any $N\in\NN$, replacing $\wi \Pi$
by $\wi \Pi^{(N)}$, we define similarly as above the quantities $\theta_k^{D,(N)}$ and $\theta^{(k,N)}$. Due
to the remark after the proof of Proposition~\ref{compnorm}, we are again brought back to the finite
framework.  The following result is an immediate consequence of Lemma~\ref{PiPi} and
Proposition~\ref{compnorm}.
\begin{lem}
We have
\begin{align*}
  \forall k\geq 2,\qquad\lim_{N\ri\iy}\theta_k^{D,(N)} & =\theta_k^{D} \\
  \forall k\geq 1,\qquad\lim_{N\ri\iy}\theta^{(k,N)} & =\theta^{(k)}
\end{align*}
\end{lem}
From this lemma and Theorem~\ref{thm:Dirichlet-finite} immediately follow all the inequalities in
Theorem~\ref{thm:Dirichlet-infinite} that concern $\theta^{D}_k$ for $k\geq 2$ and $\theta^{(k)}$ for $k\geq 1$.

As in the finite case, we easily deduce from the facts that $p^{(1)}_{i,j}=\frac{j}{i}p_{i,j}$ and
$\mu^{(1)}_i=2i^2\mu_i$ that $\wi\Pi_0u=\theta u$ with $u\in\LL^2(\NN,\mu)$ iff $\Pi_1 v=\theta v$ with
$v_i=u_i/i$ and $v\in\LL^2(\NN,\mu^{(1)})$. Therefore,
$$
\theta^{(1)}=\sup_{u\in\LL^2(\NN,\mu)\setminus\{0\}}\frac{\langle
  u,\wi\Pi_0\rangle_\mu}{\|u\|_\mu}.
$$
Since for all $a\in\RR$ and $v\in\LL^2(\NN,\mu)$, $\Pi_0(a\mathbf{1}+v)=a\mathbf{1}+\wi\Pi_0v$, we deduce that
$\theta^{(0)}=\sup\{1;\theta^{(1)}\}$.  Since $\wi\Pi_0$ is substochastic and reversible w.r.t.\ $\mu$, we
have for all $u\in\LL^2(\NN,\mu)$
$$
\frac{\langle
  u,\wi\Pi_0\rangle_\mu}{\|u\|_\mu}=\frac{\sum_{i,j\geq 1}\mu_i p_{i,j}u_i u_j}{\sum_{i\geq
    1}u_i^2\mu_i}\leq\frac{\sqrt{\sum_{i,j} u_i^2\mu_ip_{i,j}}\sqrt{\sum_{i,j}u_j^2\mu_i p_{i,j}}}{\sum_{i\geq
    1}u_i^2\mu_i}\leq 1.
$$
This yields
$$
\theta^{(0)}=1\geq\theta^{(1)}.
$$

In order to complete the proof, it only remains to check that $\theta^D_1=\theta^{(1)}$ and that $\theta^D_1$
is a Dirichlet eigenvalue in ${\cal S}_1^*$. As in the finite case, $\wit\Pi_1$ has the block-triangular
form~(\ref{eq:Pi^(1)}). Therefore, we obviously have $\theta^D_1\geq\theta^D_2$. In addition, any Dirichlet
eigenvalue in ${\cal S}_1^*$ which corresponds to an eigenvector in ${\cal U}_1$ which is nonzero on the set
of indices $\NN\times\{0\}$, must be an eigenvalue of $\wi\Pi_0$ corresponding to a right eigenvector in
$\LL^2(\NN,\mu)$. Now, if $u\in\LL^2(\NN,\mu)$ satisfies $\wi\Pi_0 u=\theta^{(1)}u$, then $\Pi_1
v=\theta^{(1)}v$ with $v_i=u_i/i$ and it follows from Theorem~\ref{thm:diago-infinite}~(a) that the vector
$\frac{i}{i+j}u_{i+j}$ is a right eigenvector of $\wit\Pi_1$. Since this vector obviously belongs to ${\cal
  U}_1$, we obtain that $\theta^D_1=\theta^{(1)}$ and that $\theta^D_1$ is a Dirichlet eigenvalue in ${\cal
  S}_1^*$.
\end{proof}

\section{Application to quasi-stationarity in nearly neutral finite absorbed two-dimensional Markov chains}
\label{sec:QSD}

In this section, we restrict ourselves to the finite state space case for simplicity: let ${\cal S}={\cal
  T}_N$ for some $N\in\NN$.  The first coordinate will be referred to as \textbf{type~1} and the second
coordinate as \textbf{type~2}.  Recall that the sets $\ZZ_+\times\{0\}$, $\{0\}\times\ZZ_+$ and $\{(0,0)\}$
are absorbing for the N2dMC considered above, which means that each sub-population in the model can go
extinct. This means that the transition matrix $\Pi$ has the form
\begin{equation}
  \label{eq:trans}
  \Pi=\begin{pmatrix} 1 & 0 \\ r & Q \end{pmatrix},
\end{equation}
after ordering the states as $(0,0)$ first. Ordering the states in ${\cal S}\setminus\{0\}$ as
$\{1,\ldots,N\}\times\{0\}$ first and $\{0\}\times\{1,\ldots,N\}$ second, the matrix $Q$ has the form
\begin{equation}
  \label{eq:Q-2-types}
  Q=\begin{pmatrix} Q_1 & 0 & 0 \\ 0 & Q_2 & 0 \\ R_1 & R_2 & Q_3 \end{pmatrix}
\end{equation}
where $Q_i$ ($1\leq i\leq 3$) are square matrices and $R_i$ ($1\leq i\leq 2$) rectangular matrices.

In this section, we study the problem of quasi-stationary distributions (QSD) and quasi-limiting distributions
(QLD, see the introduction) for Markov processes, not necessarily neutral, whose transition matrix has the
form~(\ref{eq:trans}--\ref{eq:Q-2-types}).  The classical case~\cite{darroch-seneta-65} for such a study is
the case when $Q$ is irreducible, which does not hold here. A general result is proved in
Subsection~\ref{sec:Yaglom-general}. Our results of Section~\ref{sec:Dirichlet} are then applied in
Subsection~\ref{sec:Yaglom-neutral} to study the quasi-limiting distribution of nearly neutral two-dimensional
Markov chains.

\subsection{Yaglom limit for general absorbing two-dimensional Mar\-kov chains}
\label{sec:Yaglom-general}

Let $(X_n,Y_n,n\geq 0)$ be a Markov chain on ${\cal S}={\cal T}_N$, with transition matrix of the
form~(\ref{eq:trans}--\ref{eq:Q-2-types}). We do not assume that this process is neutral. We call such
processes A2dMC for ``absorbed two-dimensional Markov chains''.

Under the assumption that the matrices $Q_1$, $Q_2$ and $Q_3$ are irreducible and aperiodic, Perron-Frobenius'
theory ensures the existence of a unique eigenvalue $\theta_i$ with maximal modulus for $Q_i$, which is real
and has multiplicity one. Moreover, $Q_i$ admits unique (strictly) positive right and left eigenvectors for
$\theta_i$, $u_i$ and $v_i$ respectively, normalised as $v_i\mathbf{1}=1$ and $v_iu_i=1$, where we use the
convention that $u_i$ is a row vector and $v_i$ is a column vector, and where $\mathbf{1}$ denotes the column
vector with adequate number of entries, all equal to 1. In the following result, we use the classical
identification of column vectors and measures: for example, $v_1=((v_1)_i)_{(i,0)\in{\cal
    S}\setminus\{(0,0)\}}$ is identified with $\sum_{(i,0)\in{\cal S}\setminus\{(0,0)\}}(v_1)_i\delta_i$, and
$v_3=((v_3)_{(i,j)})_{(i,j)\in{\cal S}^*}$ is identified with $\sum_{(i,j)\in{\cal
    S}\cap\NN^2}(v_3)_{(i,j)}\delta_{(i,j)}$.

With this notation, using the results of Darroch and Seneta~\cite{darroch-seneta-65}, $v_1\otimes\delta_0$ and
$\delta_0\otimes v_2$ are trivial QSDs for the Markov chain $(X,Y)$.

\begin{thm}
  \label{thm:Yaglom}
  Assume that the matrices $Q_1$, $Q_2$ and $Q_3$ are irreducible and aperiodic, $R_1\not=0$ and
  $R_2\not=0$. Then, for any $i\geq 1$ such that $(i,0)\in{\cal S}$,
  \begin{equation}
    \label{eq:Yaglom-1-type}
    \lim_{n\rightarrow+\infty}{\cal
      L}_{(i,0)}[(X_n,Y_n)\mid(X_n,Y_n)\not=(0,0)]=v_1\otimes\delta_0,
  \end{equation}
  and similarly for the initial state $(0,i)\in{\cal S}$, where ${\cal L}_{(i,j)}$ denotes the law of the
  Markov chain $(X_n,Y_n)_{n\geq 0}$ with initial consition $(i,j)$.

  Moreover, for any $(i,j)\in{\cal S}^*$,
  \begin{multline}
    \lim_{n\rightarrow+\infty}{\cal L}_{(i,j)}[(X_n,Y_n)
    \mid(X_n,Y_n)\not=(0,0)] \\ =
    \begin{cases}
      v_1\otimes\delta_0 & \text{if $\theta_1\geq\theta_3$ and
        $\theta_1>\theta_2$}, \\
      \delta_0\otimes v_2 & \text{if $\theta_2\geq\theta_3$ and
        $\theta_2>\theta_1$}, \\
      \displaystyle{\frac{v_3 +w_1\otimes\delta_0
      +\delta_0\otimes w_2}{1+w_1\mathbf{1}+w_2\mathbf{1}}} & \text{if
        $\theta_3>\theta_1,\theta_2$}, \\
      p_{i,j}\:v_1\otimes\delta_0+(1-p_{i,j})\:\delta_0\otimes v_2 &
      \text{if $\theta_1=\theta_2>\theta_3$}, \\
      q\:v_1\otimes\delta_0+(1-q)\:\delta_0\otimes v_2 &
      \text{if $\theta_1=\theta_2=\theta_3$}.
    \end{cases}
    \label{eq:Yaglom}
  \end{multline}
  where
  \begin{gather}
    w_i=v_3R_i(\theta_3I-Q_i)^{-1},\qquad i=1,2, \\
    p_{i,j}=\frac{\delta_{(i,j)}(\theta_1I-Q_3)^{-1}\:R_1u_1}
    {\delta_{(i,j)}(\theta_1I-Q_3)^{-1}\:(R_1u_1+R_2u_2)} \\
    \mbox{\textup{and}}\qquad q=\frac{v_3R_1u_1}{v_3(R_1u_1+R_2u_2)}.
  \end{gather}
\end{thm}

To give an interpretation of this result, we will say that there is \textbf{extinction} of type $i$
conditionally on non-extinction if the QLD~(\ref{eq:Yaglom}) gives mass 0 to all states with positive $i$-th
coordinate. Conversely, we say that there is \textbf{coexistence} conditionally on non-extinction if the
QLD~(\ref{eq:Yaglom}) gives positive mass to the set ${\cal S}^*$. We also say that type 1 is
\textbf{stronger} than type 2 (or type 2 is \textbf{weaker} than type 1) if $\theta_1>\theta_2$, and
conversely if $\theta_2>\theta_1$.

Theorem~\ref{thm:Yaglom} says that the limit behaviour of the population conditionally on non-extinction
essentially depends on whether the largest eigenvalue of $Q$ is $\theta_1$, $\theta_2$ or $\theta_3$. If
either $\theta_1>\theta_2$ and $\theta_1\geq\theta_3$ or $\theta_2>\theta_1$ and $\theta_2\geq\theta_3$, the
QLD is the same as if there were no individual of the weaker type in the initial population, and there is
extinction of the weaker type conditionally on non-extinction. If $\theta_3>\theta_1,\theta_2$, there is
\textbf{coexistence} of both types conditionally on non-extinction. Finally, when $\theta_1=\theta_2$, both
types can survive under the QLD, so none of the types go extinct (according to the previous terminology), but
there is no coexistence, as one (random) type eventually goes extinct. Observe also that the case
$\theta_1=\theta_2>\theta_3$ is the only one where the QLD depends on the initial condition.

Note also that, in the case where $\theta_3\leq\max\{\theta_1,\theta_2\}$ and $\theta_1\not=\theta_2$, the QLD
does not depend on any further information about the matrix $Q_3$. In other words, if one knows \emph{a
  priori} that $\theta_3\leq\max\{\theta_1,\theta_2\}$ and $\theta_1\not=\theta_2$, the precise transition
probabilities of the Markov chain from any state in ${\cal S}^*$ have no influence on the QLD. The QLD is only
determined by the monotype chains of types 1 and 2.

Our next result says that, for any values of $\theta_1$, $\theta_2$ and $\theta_3$, all the QSDs of the Markov
chain are those given in the r.h.s.\ of~(\ref{eq:Yaglom}), when they exist and are nonnegative.

\begin{prop}
  \label{prop:QSD}  
  Under the same assumptions and notation as in Theorem~\ref{thm:Yaglom}, the set of QSDs of the Markov chain
  is composed of the probability measures $p\:v_1\otimes\delta_0+(1-p)\:\delta_0\otimes v_2$ for all
  $p\in[0,1]$, with the additional QSD
  \begin{equation}
    \label{eq:only-other-QSD}
    \frac{v_3 +w_1\otimes\delta_0
      +\delta_0\otimes w_2}{1+w_1\mathbf{1}+w_2\mathbf{1}}
  \end{equation}
  in the case where $\theta_3>\max\{\theta_1,\theta_2\}$.
\end{prop}

\begin{proof}
The fact that all the QSDs giving no mass to the set ${\cal S}^*$ are of the form
$p\:v_1\otimes\delta_0+(1-p)\:\delta_0\otimes v_2$ for some $p\in[0,1]$ is an immediate consequence of the
facts that the sets $\{1,\ldots,N\}\times\{0\}$ and $\{0\}\times\{1,\ldots,N\}$ do not communicate and the
only QSD of an irreducible and aperiodic Markov chain on a finite set is given by the only positive
normalized left eigenvector of the transition matrix of the Markov chain (cf.~\cite{darroch-seneta-65}).

Assume now that $\mu$ is a QSD for the Markov chain $(X_n,Y_n)_{n\geq 0}$ such that $\mu({\cal S}^*)>0$, and
write $\mu=(\mu_1,\mu_2,\mu_3)$, where $\mu_1$ (resp.\ $\mu_2$, resp.\ $\mu_3$) is the restriction of $\mu$ to
the set $\{1,\ldots,N\}\times\{0\}$ (resp.\ $\{0\}\times\{1,\ldots,N\}$, resp.\ ${\cal S}^*$). The equation
$\mu Q=\theta Q$ for some $\theta>0$, which characterizes QSDs, implies that $\mu_3$ is a nonnegative left
eigenvector for $Q_3$. Thus, by the Perron-Frobenius theorem, $\mu_3=a v_3$ for some $a>0$ and
$\theta=\theta_3$. Using again the formula $\mu Q=\theta_3 Q$, one necessarily has
\begin{equation}
  \label{eq:pf-only-QSD}
  \mu_i(\theta_3\mbox{Id}-Q_i)=a v_3 R_i,\quad i=1,2.  
\end{equation}

In the case where $\theta_3>\max\{\theta_1,\theta_2\}$, the matrices $\theta_3\mbox{Id}-Q_1$ and
$\theta_3\mbox{Id}-Q_1$ are invertible, as shown in Lemma~\ref{lem:Q^n} below. Thus $\mu$ is given
by~(\ref{eq:only-other-QSD}).

In the case where $\theta_3\leq\theta_i$ for $i=1$ or $2$, we deduce from~(\ref{eq:pf-only-QSD}) that
$(\theta_3-\theta_i)\mu_iu_i=av_3R_iu_i$. This is impossible since the l.h.s.\ of this formula is non-positive
and the r.h.s.\ is positive as $R_i\not=0$, $v_3>0$ and $u_i>0$.
\end{proof}

\paragraph{Proof of Theorem~\ref{thm:Yaglom}}
For all $(k,l)\in{\cal S}\setminus\{(0,0)\}$, we want to compute the limit of
\begin{equation}
  \label{eq:condit}
  \mathbb{P}_{(i,j)}[(X_n,Y_n)=(k,l)\mid(X_n,Y_n)\not=(0,0)]
  =\frac{Q^{(n)}_{(i,j),(k,l)}}{\sum_{(k',l')\not=(0,0)}Q^{(n)}_{(i,j),(k',l')}}
\end{equation}
as $n\rightarrow+\infty$, where $Q^{(n)}_{(i,j),(k,l)}$ denotes the element of $Q^n$ on the line corresponding
to state $(i,j)\in{\cal S}$ and the column corresponding to state $(k,l)\in{\cal S}$.

Therefore, we need to analyse the behaviour of $Q^n$ as $n\rightarrow+\infty$. We have by
induction
\begin{equation}
  \label{eq:Q^n}
  Q^n=\begin{pmatrix} Q_1^n & 0 & 0 \\ 0 & Q_2^n & 0 \\ R_1^{(n)} &
    R_2^{(n)} & Q_3^n \end{pmatrix}
\end{equation}
where
\begin{equation}
  \label{eq:def-R^(n)}
  R_i^{(n)}=\sum_{k=0}^{n-1}Q_3^kR_iQ_i^{n-1-k}, \quad i=1,2.
\end{equation}
By the Perron-Frobenius Theorem (see e.g.~\cite{gantmacher-86}),
\begin{equation}
  \label{eq:Perron-Frobenius}
  Q_i^n=\theta_i^nu_iv_i+O((\theta_i\alpha)^n\mathbf{1})
\end{equation}
for some $\alpha<1$ as $n\rightarrow+\infty$, where $\mathbf{1}$ denotes the square matrix of appropriate
dimension, whose entries are all equal to 1. We need the following result. Its proof is postponed at the end
of the subsection.
\begin{lem}
  \label{lem:Q^n}
  If $\theta_1>\theta_3$, the matrix $\theta_1\text{Id}-Q_3$ is invertible, its inverse has positive entries
  and
  \begin{equation}
    \label{eq:R^(n)-1}
    R_1^{(n)}\sim\theta_1^n(\theta_1\text{Id}-Q_3)^{-1}R_1u_1v_1
  \end{equation}
  as $n\rightarrow+\infty$. If $\theta_1<\theta_3$, the matrix $\theta_3\text{Id}-Q_1$ is invertible, its
  inverse has positive entries and
  \begin{equation}
    \label{eq:R^(n)-2}
    R_1^{(n)}\sim\theta_3^nu_3v_3R_1(\theta_3\text{Id}-Q_1)^{-1}
  \end{equation}
  as $n\rightarrow+\infty$. If $\theta_1=\theta_3$, as $n\rightarrow+\infty$,
  \begin{equation}
    \label{eq:R^(n)-3}
    R_1^{(n)}\sim n\theta_1^{n-1}u_3v_3R_1u_1v_1.
  \end{equation}
\end{lem}

Theorem~\ref{thm:Yaglom} can be proved from this result as follows. Let $D$ denote the denominator
of~(\ref{eq:condit}). If $\theta_1,\theta_2>\theta_3$,~(\ref{eq:R^(n)-1})
and~(\ref{eq:Perron-Frobenius}) yield for all $(i,j)\in{\cal S}^*$
$$
D\sim\theta_1^n\delta_{(i,j)}(\theta_1\text{Id}-Q_3)^{-1}R_1u_1v_1\mathbf{1}
+\theta_2^n\delta_{(i,j)}(\theta_2\text{Id}-Q_3)^{-1}R_2u_2v_2\mathbf{1}
+\theta_3^n\delta_{(i,j)}u_3v_3\mathbf{1}
$$
as $n\rightarrow+\infty$. In the case when $\theta_1>\theta_2$, since $(\theta_1\text{Id}-Q_3)^{-1}$ has
positive entries, we have $D\sim\theta_1^n\delta_{(i,j)}(\theta_1\text{Id}-Q_3)^{-1}R_1u_1$. The limit
of~(\ref{eq:condit}) when $n\rightarrow+\infty$ then follows from~(\ref{eq:R^(n)-1}). The case
$\theta_2>\theta_1$ is treated similarly. In the case when $\theta_1=\theta_2$,
$$
D\sim\theta_1^n\delta_{(i,j)}\big[(\theta_1\text{Id}-Q_3)^{-1}R_1u_1+(\theta_2\text{Id}-Q_3)^{-1}R_2u_2\big],
$$
and the fourth line of~(\ref{eq:Yaglom}) follows from Lemma~\ref{lem:Q^n}.

In the case when $\theta_3>\theta_1,\theta_2$, we obtain
$$
D\sim\theta_3^n\delta_{(i,j)}u_3v_3\big[R_1(\theta_3\text{Id}-Q_1)^{-1}\mathbf{1}
+R_2(\theta_3\text{Id}-Q_2)^{-1}\mathbf{1}+\mathbf{1}\big],
$$
which implies the third line of~(\ref{eq:Yaglom}).

Similarly, it follows from~(\ref{eq:Perron-Frobenius}) and Lemma~\ref{lem:Q^n} that
\begin{equation*}
  \begin{array}{ll}
    D\sim\theta_1^n\delta_{(i,j)}(\theta_1\text{Id}-Q_3)^{-1}R_1u_1 & \text{if\ }\theta_1>\theta_3>\theta_2,
    \\ D\sim n\theta_1^{n-1}\delta_{(i,j)}u_3v_3R_1u_1 & \text{if\ }\theta_1=\theta_3>\theta_2, \\
    D\sim n\theta_1^{n-1}\delta_{(i,j)}u_3v_3(R_1u_1+R_2u_2) & \text{if\ }\theta_1=\theta_2=\theta_3.
  \end{array}
\end{equation*}
The proof is easily completed in each of these cases.\hfill$\Box$\bigskip

\paragraph{Proof of Lemma~\ref{lem:Q^n}}
Assume first that $\theta_1>\theta_3$. One easily checks that
\begin{equation}
  \label{eq:step-1}
  (\theta_1\text{Id}-Q_3)\sum_{k=0}^{n-1}\theta_1^{-k}Q_3^{k}=\theta_1\text{Id}-\theta_1^{-n+1}Q_3^n.  
\end{equation}
Because of~(\ref{eq:Perron-Frobenius}), the series in the previous equation converges when
$n\rightarrow+\infty$. Therefore, $\theta_1\text{Id}-Q_3$ is invertible and
$$
(\theta_1\text{Id}-Q_3)^{-1}=\sum_{n\geq 0}\theta_1^{-n-1}Q_3^n,
$$
which has positive entries since $Q_3$ is irreducible. Therefore, it follows from~(\ref{eq:Perron-Frobenius})
and~(\ref{eq:step-1}) that
\begin{align*}
  R_1^{(n)} & =\sum_{k=0}^{n-1}Q_3^kR_1\big[\theta_1^{n-1-k}u_1v_1+O((\theta_1\alpha)^{n-1-k}\mathbf{1})\big]
  \\ & =\theta_1^{n-1}(\theta_1\text{Id}-\theta_1^{-n+1}Q_3^n)(\theta_1\text{Id}-Q_3)^{-1}R_1u_1v_1
  +O\left(u_3v_3R_1\mathbf{1}\sum_{k=0}^{n-1}\theta_3^k(\theta_1\alpha)^{n-1-k}\right) \\ 
  & =\theta_1^{n}(\theta_1\text{Id}-Q_3)^{-1}R_1u_1v_1+O(\theta_3^n\mathbf{1})
  +O\left(\left(\theta_3^n+(\theta_1\alpha)^n\right)\mathbf{1}\right),
\end{align*}
where we used the fact that $\alpha$ may be increased without loss of generality so that
$\theta_3\not=\theta_1\alpha$, in which case
$$
\sum_{k=0}^{n-1}\theta_3^k(\theta_1\alpha)^{n-1-k}=\frac{\theta_3^n-(\theta_1\alpha)^n}{\theta_3-\theta_1\alpha}.
$$
Since $R_1\not=0$ has nonnegative entries and $(\theta_1\text{Id}-Q_3)^{-1}$ and $u_1v_1$ have positive
entries, the matrix $(\theta_1\text{Id}-Q_3)^{-1}R_1u_1v_1$ also has positive entries, and~(\ref{eq:R^(n)-1})
follows. The case $\theta_3>\theta_1$ can be handled similarly.

Assume finally that $\theta_1=\theta_3$. By~(\ref{eq:Perron-Frobenius}),
\begin{align*}
  R^{(n)}_1 & =\sum_{k=0}^{n-1}\left(\theta_3^ku_3v_3+O((\theta_3\alpha)^k\mathbf{1})\right)R_1
  \left(\theta_1^{n-k-1}u_1v_1+O((\theta_1\alpha)^{n-k-1}\mathbf{1})\right) \\
  & =n\theta_1^{n-1}u_3v_3R_1u_1v_1+O\left(\frac{\theta_1^{n-1}}{1-\alpha}
    \big(\mathbf{1}R_1u_1v_1+u_3v_3R_1\mathbf{1}\big)\right)+O(n(\alpha\theta_1)^{n-1}\mathbf{1}),
\end{align*}
which ends the proof of Lemma~\ref{lem:Q^n}.\hfill$\Box$\bigskip

\subsection{The nearly neutral case}
\label{sec:Yaglom-neutral}

Since $\Pi$ in~(\ref{eq:trans}) is a block triangular matrix, we have
$$
\text{Sp}'(\Pi)=\{1\}\cup\text{Sp}'(Q_1)\cup\text{Sp}'(Q_2)\cup\text{Sp}'(Q_3),
$$
where $\text{Sp}'(A)$ denotes the spectrum of the matrix $A$, where eigenvalues are counted with their
multiplicity.

In the case of a N2dMC satisfying the assumptions of Theorem~\ref{thm:Yaglom}, with the notation of
Section~\ref{sec:diago}, we have $Q_1=Q_2=\wi{\Pi}_0$. By Theorem~\ref{thm:diago-finite}~(b) and
Remark~\ref{rem:P_1}, $\{1\}\cup\text{Sp}'(Q_1)\cup\text{Sp}'(Q_2)$ is the set of eigenvalues corresponding to
right eigenvectors of $\Pi$ of the form $P_0(i,j)u_{i+j}$ and $P_1(i,j)u_{i+j}$, counted with their
multiplicity. More precisely, $\text{Sp}'(Q_1)$ corresponds to eigenvectors of the form
$P^{(1)}_1(i,j)u_{i+j}$, and $\text{Sp}'(Q_2)$ to eigenvectors of the form $P^{(2)}_1(i,j)u_{i+j}$. In
particular, $\theta_1=\theta_2=\theta^{(1)}=\theta^D_1$, with the notation of
Theorem~\ref{thm:Dirichlet-finite}. Moreover, since $Q_3=\Pi^{(2)}$, Theorem~\ref{thm:Dirichlet-finite} shows
that $\theta_3=\theta^{(2)}=\theta^D_2<\theta_1=\theta_2$ and $\text{Sp}'(Q_3)$ is the set of eigenvalues
corresponding to right eigenvectors of $\Pi$ of the form $P_d(i,j)u_{i+j}$ for $d\geq 2$, counted with their
multiplicity.

In other words, with the terminology defined after Theorem~\ref{thm:Yaglom}, \textbf{coexistence is impossible
  in the neutral case}. Since the eigenvalues of $Q_1,\ Q_2$ and $Q_3$ depend continuously on the entries of
these matrices, we deduce that \textbf{coexistence is impossible in the neighborhood of neutrality}:

\begin{cor}
  \label{cor:nearly-neutral}
  Let $\Pi$ be the transition matrix of some fixed N2dMC in ${\cal T}_N$ such that $\wi{\Pi}_0$ and
  $\wi{\Pi}_1$ are both irreducible and there exists $i\in\{1,\ldots,N\}$ such that $p_{i,0}>0$. For any
  A2dMC $(X,Y)$ in ${\cal T}_N$ with transition matrix $\Pi'$ sufficiently close to $\Pi$, coexistence is
  impossible in the QLD of $(X,Y)$. Let $\theta'_1,\theta'_2,\theta'_3$ denote the eigenvalues
  $\theta_1,\theta_2,\theta_3$ of Theorem~\ref{thm:Yaglom} corresponding to the matrix $\Pi'$. If
  $\theta'_1\not=\theta'_2$, the QLD of $(X,Y)$ is the trivial QSD corresponding to the stronger type: if
  $\theta'_1>\theta'_2$, the QLD of $(X,Y)$ is $v'_1\otimes\delta_0$, where $v'_1$ is the QLD of $(X,0)$, and
  if $\theta'_2>\theta'_1$, the QLD of $(X,Y)$ is $\delta_0\otimes v'_2$, where $v'_2$ is the QLD of $(0,Y)$.
\end{cor}

\appendix

\section{Notations}
\label{sec:nota}

We gather here all the notations used at several places in the paper. Most of these notations are introduced
in Sections~\ref{sec:diago} and~\ref{sec:Dirichlet}.

\subsection{General definitions}

\begin{itemize}
\item For any measurable subset $\Gamma$ of $\RR^d$ and any $\sigma$-finite positive measure $\mu$ on
  $\Gamma$, $\LL^2(\Gamma,\mu)$ is the set of Borel functions $f$ on $\Gamma$ defined up to $\mu$-negligible
  set such that $\int_{\Gamma}f^2d\mu<+\infty$. We denote by $\langle\cdot,\cdot\rangle_\mu$ the canonical
  inner product on $\LL^2(\Gamma,\mu)$ and $\|\cdot\|_\mu$ the associated norm. In the case when $\Gamma$ is
  discrete, we make the usual abuse of notation to identify the measure $\mu$ and the corresponding function
  on $\Gamma$.
\item For all $I\times I$ square matrix $M$, where $I$ is a finite or denumerable set of indices, and for all
  $J\subset I$, we call ``restriction of the matrix $M$ to $J$'' the matrix obtained from $M$ by removing all
  rows and columns corresponding to indices in $I\setminus J$.
\end{itemize}

\subsection{Polynomials}

$H_2(X),H_3(X),\ldots$ are defined in Proposition~\ref{prop:poly-ortho}. \\
$P_0(X,Y)=1$.\\
$P_1^{(1)}(X,Y)=P_1(X,Y)=X$.\\
$P_1^{(2)}(X,Y)=Y$.\\
$P_2(X,Y),P_3(X,Y),\ldots$ are defined in Theorem~\ref{thm:poly}.

\subsection{Sets}

$\ZZ_+=\{0,1,\ldots\}$.\\
$\NN=\{1,2,\ldots\}$.\\
$\NN_d=\{d,d+1,\ldots\}$.\\
${\cal T}_N=\{(i,j)\in\ZZ^2_+:i+j\leq N\}$, where $N\geq 0$ is fixed below.\\
${\cal T}_N^*={\cal T}_N\cap\NN^2$.

\bigskip
\noindent\begin{minipage}[h]{0.5\linewidth}
  \begin{center}
    {\bf Finite case}
  \end{center}
  ${\cal S}_Z=\{0,1,\ldots,N\}$.\\
  ${\cal S}_Z^*=\{1,2,\ldots,N\}$.\\
  ${\cal S}={\cal T}_N$.\\
  ${\cal S}^*={\cal T}_N^*$.\\
  ${\cal S}^*_k=\{(i,j)\in\ZZ^2_+:k\leq i+j\leq N\}$, for all $k\geq 2$.\\
  ${\cal S}^*_1={\cal T}_N\cap(\NN\times\ZZ_+)$.
\end{minipage}\hspace{5mm}
\begin{minipage}[h]{0.5\linewidth}
  \begin{center}
    {\bf Infinite case}
  \end{center}
  ${\cal S}_Z=\ZZ_+$.\\
  ${\cal S}_Z^*=\NN$.\\
  ${\cal S}=\ZZ_+^2$.\\
  ${\cal S}^*=\NN^2$.\\
  ${\cal S}^*_k=\{(i,j)\in\ZZ^2_+:k\leq i+j\}$, for all $k\geq 2$.\\
  ${\cal S}^*_1=\NN\times\ZZ_+$.
\end{minipage}

\subsection{Matrices}

$\Pi_0=(p_{n,m})_{n,m\in{\cal S}_Z}$ is a stochastic matrix such that $p_{0,0}=1$.\\
$\wi\Pi_{k}$ is the restriction of $\Pi_0$ to the set of indices ${\cal S}_Z\cap\NN_{k+1}$, for all $k\geq
0$.\\
$\text{Pr}_N$ is the projection operator on $\RR^\NN$ defined by
$\text{Pr}_N(u_1,u_2,\ldots)=(u_1,\ldots,u_N,0,\ldots)$.\\
$\wi\Pi_0^{(N)}=\text{Pr}_N\wi\Pi_0\text{Pr}_N$, in the infinite case (i.e.\ when ${\cal S}_Z=\ZZ_+$).\\
$\Pi_0^{(N)}$ is the Markovian kernel on $\ZZ_+$ whose restriction to $\NN$ is $\wi\Pi_0^{(N)}$.\\
$\wit\Pi_0^{(N)}$ is the restriction of $\wi\Pi_0^{(N)}$ to $\{1,\ldots,N\}$.\\
\medskip
$\Pi=(\pi_{(i,j),(k,l)})_{(i,j),(k,l)\in{\cal S}}$, where\\
$\displaystyle{
\pi_{(i,j),(i+k,j+l)}=
\begin{cases}
  \displaystyle{\frac{\binom{i+k-1}{k}\binom{j+l-1}{l}}{\binom{i+j+k+l-1}{k+l}}\:p_{i+j,\:i+j+k+l}} &
  \text{if\ }(k,l)\in\ZZ_+^2\setminus\{0\}, \\
  \displaystyle{\frac{\binom{i}{k}\binom{j}{l}}{\binom{i+j}{k+l}}\:p_{i+j,\:i+j-k-l}} & \text{if\
  }(-k,-l)\in\ZZ_+^2,\\
  0 & \text{otherwise,}
\end{cases}
}$\\
with the convention that $\binom{i}{j}=0$ if $i<0$, $j<0$ or $j>i$.\\
$\wi\Pi$ is the restriction of $\Pi$ to ${\cal S}^*$.\\
$\wit\Pi_k$ is the restriction of the matrix $\Pi$ to ${\cal S}_k^*$, for all $k\geq 1$.\\
$\wi\Pi^{(N)}$ is constructed from $\Pi_0^{(N)}$ exactly as $\wi\Pi$ is defined from $\Pi_0$.\\
$\wit\Pi^{(N)}$ is the restriction of $\wi\Pi^{(N)}$ to ${\cal T}_N^*$.\\
$\Pi_d=(p^{(d)}_{n,m})_{(n,m)\in{\cal S},\,n,m\geq d}$ for all $d\geq 0$, where for all $(n,m)\in{\cal S}$,
$n,m\geq d$,\\[2mm]
$\displaystyle{p^{(d)}_{n,m}=
  \begin{cases}
    \displaystyle{\frac{\binom{m+d-1}{m-n}}{\binom{m-1}{m-n}}\:p_{n,m}} & \text{if\ }m>n, \\
    \displaystyle{\frac{\binom{n-d}{n-m}}{\binom{n}{n-m}}\:p_{n,m}} & \text{if\ }m<n, \\
    p_{n,n} & \text{if\ }m=n.
  \end{cases}
}$\bigskip\\
{\bf For A2dMC} (see Section~\ref{sec:Yaglom-general})\\
$\displaystyle{\Pi=\begin{pmatrix} 1 & 0 \\ r & Q \end{pmatrix}}$, where
$\displaystyle{Q=\begin{pmatrix} Q_1 & 0 & 0 \\ 0 & Q_2 & 0 \\ R_1 & R_2 & Q_3 \end{pmatrix}}$.

\subsection{Measures and vectors}

$\mu=(\mu_i)_{i\in{\cal S}^*_Z}$ is a reversible measure for the matrix $\wi\Pi_0$.\\
$\nu=(\nu_{(i,j)})_{(i,j)\in{\cal S}^*}$, where $\displaystyle{\nu_{(i,j)}=\frac{(i+j)\mu_{i+j}}{ij}}$.\\
$\mu^{(d)}=(\mu^{(d)}_n)_{n\in{\cal S}_Z\cap\NN_d}$ for all $d\geq 1$, where\\
$\displaystyle{\mu_n^{(d)}= 2\,n\,\binom{n+d-1}{2d-1}\,\mu_n}$ for all $n\in{\cal
  S}_Z\cap\NN_d$.\\
$g^{(d)}=(g^{(d)}_i)_{i\in{\cal S}_Z\cap\NN_{d+1}}$ for all $d\geq 0$, where\\
$\displaystyle{g^{(0)}_i=2i^2}$ for all $i\in{\cal S}_Z^*$ and\\
$\displaystyle{g^{(d)}_i=\frac{(i-d)(i+d)}{2d(2d+1)}}$ for all $d\geq 1$ and $i\in{\cal S}_Z\cap\NN_{d+1}$.\\
$h^{(d)}=(h^{(d)}_{i,j})_{(i,j)\in{\cal S},\,i,j\geq d+1}$ for all $d\geq 0$, where\\
$\displaystyle{h^{(0)}_{i,j}=\frac{j}{i}}$ for all $(i,j)\in{\cal S}^*$ and\\[2mm]
$\displaystyle{h^{(d)}_{i,j}=
  \begin{cases}
    \displaystyle{\frac{j+d}{i+d} }& \text{if\ }j>i, \\
    \displaystyle{\frac{j-d}{i-d}} & \text{if\ }j<i, \\
    1 & \text{if\ }i=j.
  \end{cases}
}$ $\ $ for all $d\geq 1$ and $(i,j)\in{\cal S}$, $i,j\geq d+1$.

\subsection{Operator norms (in the infinite, reversible case)}

$\vvv\cdot\vvv_0$ is the natural operator norm on the set of bounded operators on $\LL^2(\NN,\mu)$.\\
$\vvv\cdot\vvv$ is the natural operator norm on the set of bounded operators on $\LL^2(\NN^2,\nu)$.\\
$\vvv\cdot\vvv_d$ is the natural operator norm on the set of bounded operators on $\LL^2(\NN_d,\mu^{(d)})$,
for all $d\geq 1$.

\subsection{Eigenvalues}

We refer to Sections~\ref{sec:finite-2} and~\ref{sec:infinite-2} for precise definitions in the finite and
infinite cases.
\bigskip

\noindent $\theta^D_k$ is the biggest Dirichlet eigenvalue of $\Pi$ in ${\cal S}^*_k$, for all $k\geq 1$.\\
$\theta^{(d)}$ is the biggest eigenvalue of $\Pi$ corresponding to right eigenvectors of the form
$P_d(i,j)u_{i+j}$, for all $d\geq 0$.\bigskip\\
{\bf For A2dMC} (see Section~\ref{sec:Yaglom-general})\\
$\theta_1$, $\theta_2$ and $\theta_3$ are the Perron-Fr\"obenius eigenvalues of $Q_1$, $Q_2$ and $Q_3$,
respectively.

\subsection{Vector spaces}

${\cal V}_d:=\{v\in\RR^{\cal S}:v_{i,j}=P_d(i,j)u_{i+j}\text{\ with\ }u\in\RR^{{\cal S}_Z}\}$.\bigskip\\
{\bf Finite case}\\
${\cal U}_k=\{v\in\RR^{{\cal T}_N}:v=0\text{\ on\ }{\cal T}_N\setminus{\cal S}^*_k\}$ for all $k\geq 1$.\bigskip\\
{\bf Infinite, reversible case}\\
${\cal U}_k=\{v\in\LL^2(\ZZ_+^2,\nu):v=0\text{\ on\ }\ZZ_+^2\setminus{\cal S}^*_k\}$ for all $k\geq 2$.\\
${\cal U}_1=\Big\{v\in\RR^{\ZZ_+^2}:v_{i,j}=\frac{i}{i+j}v^{(1)}_{i+j}+v^{(3)}_{i,j},\
v^{(1)}\in\LL^2(\NN,\mu),\ v^{(3)}\in\LL^2(\NN^2,\nu)\Big\}$.\\
${\cal V}'_d=\left\{v\in\LL^2(\NN^2,\nu):v_{i,j}=P_d(i,j)u_{i+j}\text{\ with\
  }u\in\LL^2(\NN_d,\mu^{(d)})\right\}$ for all $d\geq 2$.

\bigskip

\noindent{\bf Acknowledgments:} We are grateful to Pierre-Emmanuel Jabin for helpful
discussions on the continuous case of Section~\ref{sec:cont} and to Bob Griffiths for his
comments on a preliminary version of this work and for pointing out the
reference~\cite{karlin-mcgregor-75}.\\ Nicolas Champagnat's research was partly supported
by project MANEGE `Mod\`eles Al\'eatoires en \'Ecologie, G\'en\'etique et \'Evolution'
09-BLAN-0215 of ANR (French national research agency).

\bibliographystyle{plain}
\bibliography{biblio-bio,biblio-math}

\end{document}